\documentclass[graybox]{svmult}

\usepackage{amsmath}
\usepackage{amssymb}
\usepackage{epstopdf}
\usepackage{mathptmx}       
\usepackage{helvet}         
\usepackage{courier}        
\usepackage{type1cm}        
\usepackage{makeidx}         
\usepackage{graphicx}        
\usepackage{multicol}        
\usepackage[bottom]{footmisc}
\makeindex             

\newtheorem{prop}{Proposition}
\newcommand\floor[1]{\lfloor#1\rfloor}

\begin{document}

\title*{Multi-Grid Schemes for Multi-Scale Coordination of Energy Systems}
\author{Sungho Shin and Victor M. Zavala}
\institute{Sungho Shin \at Department of Chemical and Biological Engineering, University of Wisconsin-Madison, 1415 Engineering Dr, Madison, WI 53706, USA \email{shin79@wisc.edu}
\and Victor M. Zavala \at Department of Chemical and Biological Engineering, University of Wisconsin-Madison, 1415 Engineering Dr, Madison, WI 53706, USA \email{zavalatejeda@wisc.edu}}
%
%
\maketitle

\vspace{-1in}

\abstract{We discuss how multi-grid computing schemes can be used to design hierarchical coordination architectures for energy systems. These hierarchical architectures can be used to manage multiple temporal and spatial scales and mitigate fundamental limitations of centralized and decentralized architectures. We present the basic elements of a multi-grid scheme, which includes a smoothing operator (a high-resolution decentralized coordination layer that targets phenomena at high frequencies) and a coarsening operator (a low-resolution centralized coordination layer that targets phenomena at low frequencies).  For smoothing, we extend existing convergence results for Gauss-Seidel schemes by applying them to systems that cover unstructured domains. This allows us to target problems with multiple timescales and arbitrary networks.  The proposed coordination schemes can be used to guide transactions in decentralized electricity markets. We present a storage control example and a power flow diffusion example to illustrate the developments.} 

\vspace{-0.3in}\section{Motivation and Setting}
\label{sec:1}
We consider the following optimization problem: 
\begin{subequations}\label{multi-scale control}
\begin{align}
\min_{{z}}\quad & \frac{1}{2}z^T Q z-c^Tz \\
\text{s.t.}\quad & Az+Bd={0},\quad (\nu)\label{eq:const}\\
\label{fix}&\Pi\, {z}=0,\quad\;\quad\;\;\; (\lambda)
\end{align}
\end{subequations}
Here, ${z}\in\mathbb{R}^{N\cdot n_z}$ are decision or primal variables (including states and controls) and ${d}\in\mathbb{R}^{N\cdot n_d}$ is the data (including disturbances and system parameters). These variable vectors contain elements that are distributed over a mesh with $N\in\mathbb{Z}$ points that covers a certain temporal or spatio-temporal domain of interest $\Omega$.  We define the set of points in the mesh as $\mathcal{N}$ with $|\mathcal{N}|=N$.  The matrix $Q\in \mathbb{R}^{N\cdot n_z\times N\cdot n_z}$ is positive definite and $c\in \mathbb{R}^{N\cdot n_z}$ is a cost vector. The constraint \eqref{eq:const} (with associated dual variables $\nu\in\mathbb{R}^{m}$) is defined by the matrices $A \in\mathbb{R}^{m \times N\cdot n_z}$ and $B\in \mathbb{R}^{ m\times N\cdot n_d}$ and the matrix $A$ is assumed to have full row rank. The constraints may include discretized dynamic equations (in space and time) and other physical constraints. The constraints \eqref{fix} (with associated dual variables $\lambda\in \mathbb{R}^{ p}$) are defined by the matrix $\Pi\in \mathbb{R}^{p\times N\cdot n_z}$. This constraint models coupling (connectivity) between the primal variables at different mesh points and can also be used to model boundary conditions.  Problem \eqref{multi-scale control} captures formulations used in optimization-based control strategies such as model predictive control (MPC) \cite{rawlingsbook}. 

We assume that the dimension of the mesh $N$ describing problem \eqref{multi-scale control} is so large that the problem cannot be solved in a {\em centralized} manner. This is often the case in systems that cover large temporal and spatial domains and/or multiple scales. In an electrical network, for instance,  a  large number of nodes and harmonics might need to be captured, rendering centralized control impractical. An alternative to address this complexity is to partition the problem into subdomains to create {\em decentralized} control architectures.  We begin by defining a partitioned version of problem \eqref{multi-scale control}:
\begin{subequations}\label{lifted problem}
\begin{align}
\min_{{z}_k}\quad & \sum_{k\in \mathcal{K}}\frac{1}{2}z_k^TQ_kz_k-c_k^Tz_k \\
\label{constraint 1}\text{s.t.}\quad  & A_kz_k+B_k d_k=0,\quad \,k\in \mathcal{K}\quad (\nu_k)\\
&\label{keep duplicate equal}
\sum_{k'\in \mathcal{K}}\Pi_{kk'} {z}_{k'}=0,\quad \,k\in \mathcal{K}\quad (\lambda_k)
\end{align}
\end{subequations}
We denote this problem as $\mathcal{P}$. Here, $\mathcal{K}$ is a set for partitions (subdomains) of the set $\mathcal{N}$ and we define the number of partitions as $K:=|\mathcal{K}|$. Each partition $k\in\mathcal{K}$ contains mesh elements $\mathcal{N}_k\subseteq \mathcal{N}$ satisfying $\cup_{k\in\mathcal{K}}\mathcal{N}_k=\mathcal{N}$ and $\mathcal{N}_k\cap\mathcal{N}_{k'}=\emptyset$ for all $k, k'\in\mathcal{K}$ and $k\neq k'$. The number of elements in a partition is denoted as $N_k:=|\mathcal{N}_k|$. The variables and data $({z}_k, {d}_k)$ are defined over the partition $k\in\mathcal{K}$. We represent the cost function as a sum of the partition cost functions with associated positive definite matrices $Q_k$ and cost vectors $c^k$.  The constraints are also split into individual partition constraints with associated matrices $A_k,B_k$ and we link the partition variables by using the coupling constraints \eqref{keep duplicate equal} and associated matrices $\Pi_{kk'}, \, k,k'\in \mathcal{K}$. As we will discuss later, we can always obtain such a representation by introducing duplicate decision variables in each partition and by adding suitable coupling constraints. This procedure is known as {\em lifting} \cite{diehllift}.

To {\em avoid centralized coordination}, a wide range of {\em decomposition} schemes (we also refer to them as {\em decentralized} coordination schemes) can be used.  A popular approach used in the solution of partial differential equations (PDEs) and decomposition methods such as the alternating direction method of multipliers (ADMM) is the {\em Gauss-Seidel} (GS) coordination scheme \cite{borzi,boydadmm}. Here, the problem in each partition $k$ (often called a control {\em agent}) is solved independently from the rest and exchanges information with its neighbors to coordinate. For a lifted problem of the form \eqref{lifted problem}, we will derive a decentralized GS scheme that solves problems over individual partitions $k\in\mathcal{K}$ of the form: 
\begin{subequations}\label{iterative scheme} 
\begin{align}
z_{k}^{\ell+1}=\mathop{\textrm{argmin}}_{{z}_k}\quad & \frac{1}{2}z_k^TQ_kz_k-{{{z}}}_k^T\left(c_k-\sum^{k-1}_{k'=1}{\Pi_{k'k}}^T\lambda_{k'}^{\ell+1}-\sum^{N}_{k'=k+1}{\Pi_{k'k}}^T\lambda_{k'}^\ell\right) \\
\text{s.t.}\quad  & A_kz_k+B_kd_k={0}\\
&{\Pi}_{kk}z_k+\sum_{k'=1}^{k-1} {\Pi}_{kk'}z_{k'}^{\ell+1}+\sum_{k'=k+1}^{K} {\Pi}_{kk'}z_{k'}^{\ell}=0 \quad (\lambda_k).\label{eq:coupling2}
\end{align}
\end{subequations}
We denote this partition subproblem as $\mathcal{P}_k^\ell$ that is solved at the update step $\ell\in\mathbb{Z}_+$ (that we call here the {\em coordination step}). From the solution of this problem we obtain the updated primal variables $z_k^{\ell+1}$ and dual variables $\lambda_k^{\ell+1}$ (corresponding to the coupling constraints \eqref{eq:coupling2}). Here, $z_{k'}^{\ell}$ and $\lambda_{k'}^{\ell}$ are primal and dual variables for neighboring partitions connected to partition $k$ and that have not been updated while $z_{k'}^{\ell+1}$ and $\lambda_{k'}^{\ell+1}$ are primal and dual variables for neighboring partitions  that have already been updated. We refer to the variables communicated between partitions as the {\em coordination variables}.   We note that partition $k$ cannot update its primal and dual variables until the variables of a subset of the partitions connected to it have been updated. Consequently, the GS scheme is sequential and synchronous in nature. We also note that the connectivity topology (induced by the coupling matrices $\Pi_{kk'}$)  determines the communication structure. We highlight, however, that the order of the updates presented in \eqref{iterative scheme} is lexicographic (in the order of the partition number) but this choice of update order is arbitrary and can be modified. We will see that the update order can be designed to derive parallel schemes (i.e., in which certain partitions can proceed independently of others) but that the order can affect performance.  In Figure \ref{step figure} we illustrate the configuration of a GS scheme over a 1-D mesh while in Figure \ref{spatial figure} we present a configuration for a 2-D mesh.  For the 2-D mesh, we note that the nodes spanning the domain are grouped into sets of the form $\mathcal{N}_{m,n}$ and we note that the information is exchanged using the state and dual variables in the boundary of the partitions. In Section \ref{sec:cases} we discuss this approach in more detail. 

In the next sections we derive and analyze a decentralized GS scheme to solve problem $\mathcal{P}$. The analysis seeks to illustrate how the structure of the partition subproblem $\mathcal{P}_k$ arises and to highlight how information of the coordination variables propagates throughout the partitions. We then discuss how to create {\em coarse representations} of the full resolution problem $\mathcal{P}$ to obtain approximations for the coordination variables and with this accelerate the decentralized GS scheme. This gives rise to the concept of {\em multi-grid schemes}, that  can be used to design hierarchical coordination architectures.  Our analysis is performed on convex quadratic programs (QPs), which will reveal important features of multi-grid schemes. 
 
\begin{figure}[!htb]
\begin{center}
\includegraphics[width=4.0in]{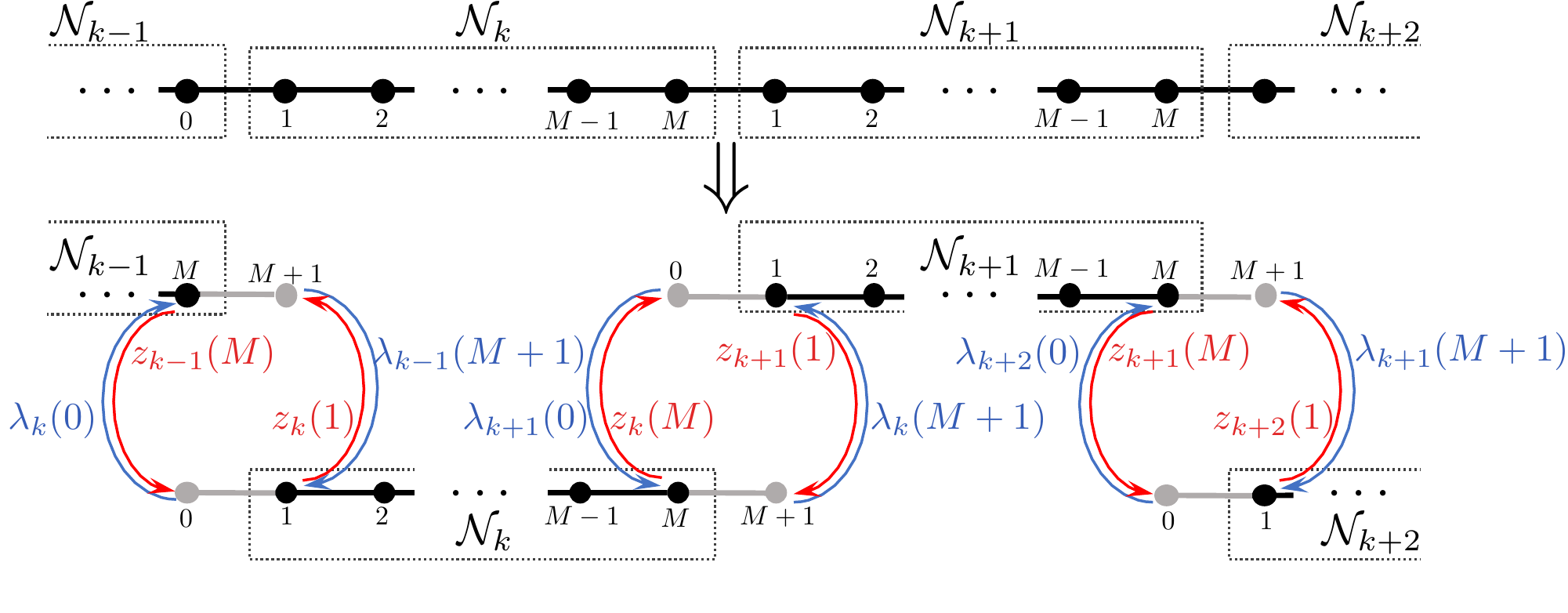}\caption{Configuration of a GS scheme over a 1-D mesh.}\label{step figure}
\end{center}
\end{figure}

\begin{figure}[!htb]
\begin{center}
\includegraphics[width=4.5in]{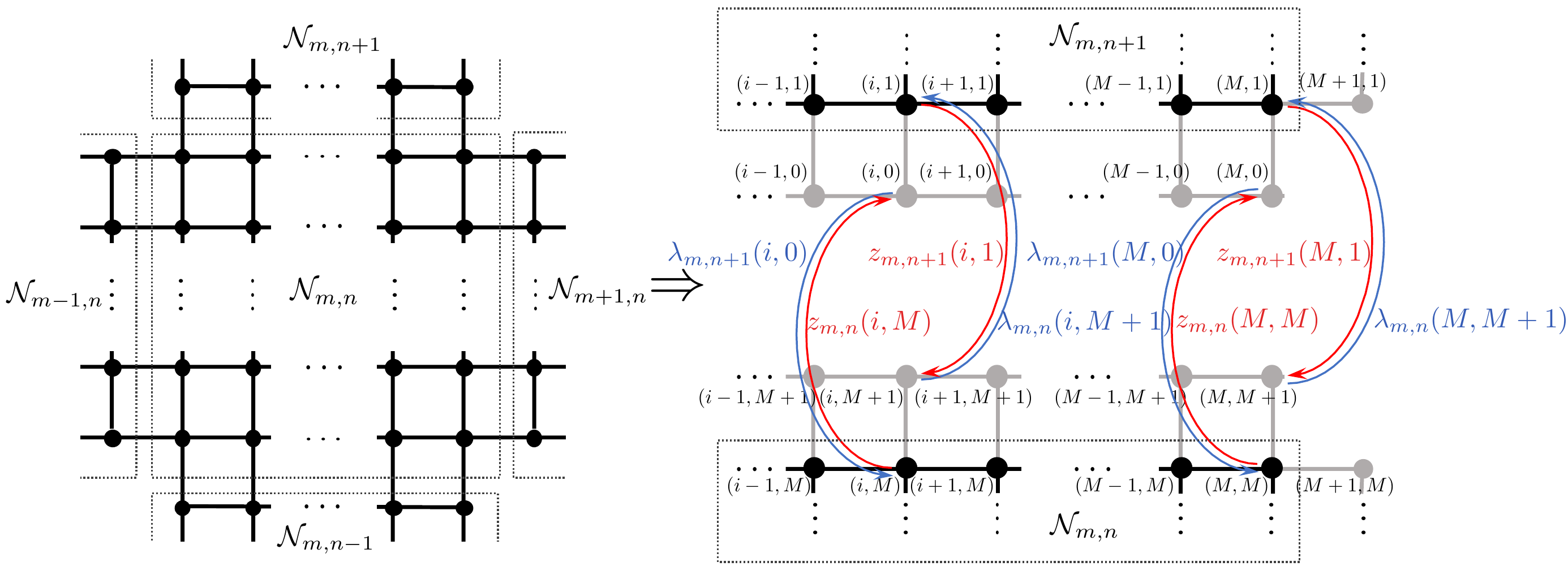}\caption{Configuration of a GS scheme over a 2-D mesh.}\label{spatial figure}
\end{center}
\end{figure}

The concepts discussed in this paper seek to extend existing literature on decentralized and hierarchical MPC. Many strategies have been proposed to address the complexity of centralized MPC such as spatial and temporal decomposition methods \cite{rawlingscooperative,kroghdecentralized,scattoreview,frasch,deschutter}, fast inexact schemes \cite{zavalaasnmpc,zavalage,diehlrealtime}, and reduced order modeling techniques \cite{daoutidis}. Decentralized control manages the entire system by coordinating multiple controllers, each operating in a different node or subnetwork. Decentralization also enables {\em resiliency and asynchronicity}, which are key practical advantages over centralized MPC. Different schemes for coordinating MPC controllers have been devised \cite{scattoreview}. Lagrangian dual decomposition is a technique where Lagrange multipliers are used for coordination. This technique is popular in electricity markets because the dual variables can be interpreted as prices that are used as a coordination mechanism \cite{illic,deschutter2,arnold}. Techniques based on coordinate minimization schemes and distributed gradient methods have also been proposed to coordinate MPC controllers in general settings \cite{rawlingsjpc,mpcadmm,liuadmm}.  An important limitation of decentralized schemes is that coordination of subsystems tends to be slow (e.g., convergence rates of existing schemes are at best linear) \cite{boydadmm,luoadmm,fisher}.  This slow convergence has been reported in the context of energy networks in \cite{arnold}. Moreover, spatial decentralization by itself does not address the complexity induced by multiple time scales. In particular, time scales and prediction horizons of different decentralized  controllers might be different. To the best of our knowledge, no coordination schemes currently exist to handle such settings. 

Hierarchical control seeks to overcome fundamental limitations of decentralized and centralized control schemes.  Fundamental concepts of hierarchical control date as far back as the origins of automatic control itself \cite{hierarbook}. Complex industrial control systems such as the power grid are structured hierarchically in one way or another to deal with multiple time and spatial scales. Existing hierarchies, however, are often constructed in ad-hoc manners by using objectives, physical models, and control formulations at different levels that are often {\em incompatible}. For instance, an independent system operator (ISO) solves a hierarchy of optimization problems (unit commitment, economic dispatch, optimal power flow) that use different physical representations of the system. This can lead to lost economic performance, unreachable/infeasible command signals, and instabilities \cite{rawlingsunreachable,zavalaipopt}. Hierarchical MPC provides a general framework to tackle dynamics and disturbances occurring at multiple time scales  \cite{scattohierar,scattoreview,hierarbook,zavalamultigrid} and spatial scales \cite{farina2017hierarchical}. In a traditional hierarchical MPC scheme, one uses a high-level controller to compute coarse control actions that are used as targets (commands) by low level controllers. This approach has been used recently in microgrids and multi-energy systems \cite{hierarresid,hug}. More sophisticated MPC controllers use robustness margins of the high level controller that are used by the lower level controller to maintain stability \cite{scattohierar}.  Significant advances in the analysis of multi-scale dynamical systems have also been made, most notably by the use of singular perturbation theory to derive reduced-order representations of complex networks \cite{kokotovicpower,chowpower,peponides,simon}. The application of such concepts in hierarchical MPC, however, has been rather limited. In particular, the recent review on hierarchical MPC by Scattolini notices that systematic design methods for hierarchical MPC are still lacking \cite{scattoreview}. More specifically, no hierarchical MPC schemes have been proposed that aggregate and {\em refine} trajectories at multiple scales. In addition, existing schemes have been tailored to achieve feasibility and stability but {\em do not have optimality guarantees}. This is important in systems where both economic efficiency and stability must be taken into account. To the best of our knowledge, no attempt has been made to combine hierarchical and decentralized MPC schemes to manage spatial and temporal scales simultaneously.  The multi-grid computing concepts presented in this work seek to take a first step towards creating more general hierarchical control architectures.   The proposed multi-grid schemes provide a framework to {\em coordinate decentralized electricity markets}. This is done by exchanging state and price (dual information) at the interfaces of the agents domain. The ability to do this hierarchically enables coordination over multiple spatial schemes, in particular, provides a framework to cover large geographical regions that might involve many market players.  
 
 \vspace{-0.3in}\section{Analysis of Gauss-Seidel Schemes}
 
This section presents basic concepts and convergence results for a GS scheme under a general convex QP setting. The results  seek to highlight how the structure of the coupling between partition variables as well as the coordination sequence affect the performance of GS schemes.  
 
\vspace{-0.2in}\subsection{Illustrative Setting}
To introduce notation, we begin by considering a convex QP with two variable partitions (i.e., $\mathcal{K}=\{1,2\}$). To simplify the presentation, we do not include internal partition constraints and focus on coupling constraints across partitions. Under these assumptions, we have the following lifted optimization problem:
\begin{subequations}\label{special problem}
\begin{align}
\min_{{z}_1,{z}_2}\quad & 
\frac{1}{2}
\begin{bmatrix}
{z}_1\\
{z}_2
\end{bmatrix}^T
\begin{bmatrix}
Q_1&\\
&Q_2
\end{bmatrix}
\begin{bmatrix}
{z}_1\\
{z}_2
\end{bmatrix}
-
\begin{bmatrix}
c_1\\
c_2
\end{bmatrix}^T
\begin{bmatrix}
{z}_1\\
{z}_2
\end{bmatrix}
 \\
\text{s.t.}\quad &
\underbrace{\begin{bmatrix}
\Pi_{11}&\Pi_{12}\\
\Pi_{21}&\Pi_{22}
\end{bmatrix}}_{\Pi}
\underbrace{\begin{bmatrix}
{z}_1\\
{z}_2
\end{bmatrix}}_{z}=
\begin{bmatrix}
0\\0
\end{bmatrix}\quad \begin{array}{c}(\lambda_1)\\ (\lambda_2) \end{array}
\end{align}
\end{subequations}
Here, the matrices $\Pi_{11},\Pi_{12}, \Pi_{21}, \Pi_{22}$ capture coupling between the variables $z_1$ and $z_2$. We highlight that, in general, $\Pi_{12}\neq \Pi_{21}$ (the coupling between partitions is not symmetric). The matrices $Q_1$ and $Q_2$ are positive definite and we assume that $\Pi_{11}$ and $\Pi_{22}$ have full row rank and that the entire coupling matrix $\Pi$ has full row rank. By positive definiteness of $Q_1$ and $Q_2$ and the full rank assumption of $\Pi$ we have that the feasible set is non-empty and the primal and dual solution of \eqref{special problem} is unique.  The coupling between partition variables is two-directional (given by the structure of the coupling matrix $\Pi$). This structure is found in 1-D linear networks. When the coupling is only one-directional, as is the case of temporal coupling, we have that $\Pi_{12}=0$ and thus $\Pi$ is block lower triangular, indicating that the primal variables only propagate forward in time. We will see, however, that backward propagation of information also exists, but in the space of the dual variables.   

The first-order KKT conditions of \eqref{special problem} are given by the linear system:
\begin{equation}\label{eq:kkt}
\begin{bmatrix}
Q_1&{\Pi_{11}}^T&&{\Pi_{21}}^T\\
\Pi_{11}&&\Pi_{12}&\\
&{\Pi_{12}}^T&Q_2&{\Pi_{22}}^T\\
\Pi_{21}&&\Pi_{22}&
\end{bmatrix}
\begin{bmatrix}
{z}_1\\
{\lambda}_1\\
{z}_2\\
{\lambda}_2
\end{bmatrix}
=
\begin{bmatrix}
c_1\\
0\\
c_2\\
0
\end{bmatrix}
\end{equation}
To avoid solving this system in a centralized manner we use a decentralized GS scheme with {\em coordination update} index $\ell\in\mathbb{Z}_+$. The scheme has the form:
\begin{subequations}
\label{step}
\begin{align}\label{step1}
\begin{bmatrix}
Q_1&{\Pi_{11}}^T\\
{\Pi_{11}}
\end{bmatrix}
\begin{bmatrix}
{z}_1^{\ell+1}\\
{\lambda}_1^{\ell+1}
\end{bmatrix}
&=
\begin{bmatrix}
c_1\\
0
\end{bmatrix}
-
\begin{bmatrix}
&{\Pi_{21}}^T\\
{\Pi_{12}}&
\end{bmatrix}
\begin{bmatrix}
{z}_2^{\ell}\\
{\lambda}_2^{\ell}
\end{bmatrix}\\
\label{step2}
\begin{bmatrix}
Q_2&{\Pi_{22}}^T\\
{\Pi_{22}}
\end{bmatrix}
\begin{bmatrix}
{z}_2^{\ell+1}\\
{\lambda}_2^{\ell+1}
\end{bmatrix}
&=
\begin{bmatrix}
c_2\\
0
\end{bmatrix}
-
\begin{bmatrix}
&{\Pi_{12}}^T\\
{\Pi_{21}}&
\end{bmatrix}
\begin{bmatrix}
{z}_1^{\ell+1}\\
{\lambda}_1^{\ell+1}
\end{bmatrix}
\end{align}
\end{subequations}

This scheme requires an initial guess for the variables of the second partition $({z}_2,{\lambda}_2)$ that is used to update $(z_1,\lambda_1)$ and we then proceed to update the variables of the second partition. Note, however, that we have picked this {\em coordination order} arbitrarily (lexicographic order). In particular, one can start with an initial guess of the first partition to update the second partition and then update the first partition (reverse lexicographic order). This scheme has the form:
\begin{subequations}
\begin{align}
\begin{bmatrix}
Q_2&{\Pi_{22}}^T\\
{\Pi_{22}}
\end{bmatrix}
\begin{bmatrix}
{z}_2^{\ell+1}\\
{\lambda}_2^{\ell+1}
\end{bmatrix}
&=
\begin{bmatrix}
c_2\\
0
\end{bmatrix}
-
\begin{bmatrix}
&{\Pi_{12}}^T\\
{\Pi_{21}}&
\end{bmatrix}
\begin{bmatrix}
{z}_1^{\ell}\\
{\lambda}_1^{\ell}
\end{bmatrix}\\ 
\begin{bmatrix}
Q_1&{\Pi_{11}}^T\\
{\Pi_{11}}
\end{bmatrix}
\begin{bmatrix}
{z}_1^{\ell+1}\\
{\lambda}_1^{\ell+1}
\end{bmatrix}
&=
\begin{bmatrix}
c_1\\
0
\end{bmatrix}
-
\begin{bmatrix}
&{\Pi_{21}}^T\\
{\Pi_{12}}&
\end{bmatrix}
\begin{bmatrix}
{z}_2^{\ell+1}\\
{\lambda}_2^{\ell+1}
\end{bmatrix}
\label{step3}
\end{align}
\end{subequations}

We will see that the coordination order affects the convergence properties of the GS scheme. In problems with many partitions, we will see that a large number of coordination orders are possible. Moreover, we will see that coordination orders can be designed to enable sophisticated parallel and asynchronous implementations.  A key observation is that the linear systems \eqref{step} in the GS scheme are the first-order KKT conditions of the following partition problems $\mathcal{P}_1$ and $\mathcal{P}_2$, respectively:
\begin{subequations}
\label{opt iter}
\begin{align}
\label{opt iter 1a}
{z}_1^{\ell+1}=\mathop{\textrm{argmin}}_{{z}_1} &\quad\frac{1}{2}z_1^T Q_1 {z}_1 - z_1^T\left({c_1}-\Pi_{21}^T {\lambda}_2^{\ell}\right) \\
\label{opt iter 1b}
\text{s.t.} &\quad \Pi_{11}{z}_1+\Pi_{12} {z}_2^{\ell}=0 \quad (\lambda_1)
\end{align}
\begin{align}
\label{opt iter 2a}
{z}_2^{\ell+1}=\mathop{\textrm{argmin}}_{{z}_2}&\quad\frac{1}{2}z_2^T Q_2 {z}_2 - z_2^T\left({c_2} -\Pi_{12} ^T {\lambda}_1^{\ell+1}\right)\\
\label{opt iter 2b}
\text{s.t.}&\quad  \Pi_{22}{z}_2+\Pi_{21} {z}_1^{\ell+1}=0 \quad (\lambda_2)
\end{align}
\end{subequations}
The relevance of this observation is that one can implement the GS scheme by {\em directly solving optimization problems}, as opposed to performing intrusive linear algebra calculations \cite{zavalamultigrid}. This has practical benefits, as one can use algebraic modeling languages and handle sophisticated problem formulations. This also reveals that {\em both primal and dual variables} are communicated between partitions. Primal information enters in the coupling constraints. The dual variables enter as cost terms in the objective and highlights the fact that dual variables can be interpreted as {\em prices of the primal variable information exchanged} between partitions. 

We now seek to establish conditions guaranteeing convergence of this simplified GS scheme. To do so, we define the following matrices and vectors:
\begin{subequations}
\begin{align}
\label{define 0}
&A_1:=
\begin{bmatrix}
Q_1&{\Pi_{11}}^T\\
{\Pi_{11}}
\end{bmatrix},\;
{x}_	1^\ell:=
\begin{bmatrix}
{z}_1^{\ell}\\
{\lambda}_1^{\ell}
\end{bmatrix}, \;
{b}_1:=
\begin{bmatrix}
c_1\\
0
\end{bmatrix},\;
B_{12}:=
\begin{bmatrix}
&-{\Pi_{21}}^T\\
-{\Pi_{12}}&
\end{bmatrix}\\
\label{define 1}
&A_2:=
\begin{bmatrix}
Q_2&{\Pi_{22}}^T\\
{\Pi_{22}}
\end{bmatrix},\;
{x}_2^\ell:=
\begin{bmatrix}{z}_2^{\ell}\\
{\lambda}_2^{\ell}
\end{bmatrix},\;
{b}_2:=
\begin{bmatrix}
c_2\\
0
\end{bmatrix},\;
B_{21}:=
\begin{bmatrix}
&-{\Pi_{12}}^T\\
-{\Pi_{21}}&
\end{bmatrix}.
\end{align}
\end{subequations}
The partition matrices $A_1$, $A_2$ are non-singular because the matrices $Q_1$, $Q_2$ are positive definite and $\Pi_{11}$, $\Pi_{22}$ have full row rank. Nonsingularity of $A_1$ and $A_2$ implies that the partition optimization subproblems $\mathcal{P}_1$ and $\mathcal{P}_2$ have a unique solution for any values of the primal and dual variables of the neighboring partition.  We can now express $({x}_1^{\ell+1},{x}_2^{\ell+1})$ in terms of $({x}_1^\ell,{x}_2^\ell)$ to obtain a recursion of the form: 
\begin{gather}
\begin{bmatrix}
A_1&&\\
-B_{21}&A_2
\end{bmatrix}
\begin{bmatrix}
{x}_1^{\ell+1}\\
{x}_2^{\ell+1}
\end{bmatrix}
=
\begin{bmatrix}
\quad&B_{12}\\
\quad & 
\end{bmatrix}
\begin{bmatrix}
{x}_1^{\ell}\\
{x}_2^{\ell}
\end{bmatrix}
+
\begin{bmatrix}
{b}_1\\
{b}_2
\end{bmatrix}.
\end{gather}
We can write this system in compact form by defining:
\begin{equation}\label{define S}
{w}^{\ell}:=
\begin{bmatrix}
{x}_1^{\ell}\\
{x}_2^{\ell}
\end{bmatrix},\quad
S:=
\begin{bmatrix}
A_1&&\\
-B_{21}&A_2
\end{bmatrix}^{-1}\begin{bmatrix}
\quad&B_{12}\\
\quad & 
\end{bmatrix}
,\quad
{r}:=\begin{bmatrix}
A_1&&\\
-B_{21}&A_2
\end{bmatrix}^{-1}\begin{bmatrix}
{b}_1\\
{b}_2
\end{bmatrix}.
\end{equation}
The solution of the $\ell$-th update step of \eqref{step} can be represented as:
\begin{subequations}
\begin{align}
{w}^{\ell+1}&=S{w}^\ell+{r}.\label{sol a}\\
&=S^\ell{w}^0+\left(I+S+\cdots+S^{\ell-1}\right){r}.\label{sol b}
\end{align}
\end{subequations}
The solution of the QP \eqref{special problem} (i.e., the solution of the KKT system \eqref{eq:kkt}) solves the implicit system ${w}=S{w}+{r}$, which can also be expressed as $(I-S)w=r$ or $w=(I-S)^{-1}r$. Consequently, we note that the eigenvalues of matrix $S$ play a key role in the convergence of  the GS scheme \eqref{opt iter}. We discuss this in more detail in the following section.

\vspace{-0.2in}\subsection{General Setting}

We now extend the previous analysis to a more general QP setting with an arbitrary number of partitions.  Here, we seek to illustrate how to perform {\em lifting} in a general case where coupling is implicit in the model and how to derive a GS scheme to solve such a problem.  Our discussion is based on the following convex QP:
\begin{equation} \label{general}
\min_{{z}}\quad  \frac{1}{2}{z}^T Q {z} -{c}^T{z}
\end{equation}
Here ${z}=(z(1),z(2),\cdots,z(N))\in \mathbb{R}^{N}$ are the optimization variables, $Q\in\mathbb{R}^{N\times N}$ is a positive definite matrix, and $c\in\mathbb{R}^{N}$ is the cost vector. For simplicity (and without loss of generality) we assume that $n_z=1$ (there is only one variable per node). We let $Q_{ij}$ represent the $(i,j)$-th component of matrix $Q$ and we let $\mathcal{N}=\{1,2,\cdots,N\}$ be the variable indices (in this case also the node indices). Problem \eqref{general} has a unique solution because the matrix $Q$ is positive definite. 

We focus our analysis on the QP \eqref{general} because we note that equality constraints $\bar{A}z+\bar{B}d=0$ in \eqref{multi-scale control} (with $\bar{A}^T=[A^T\, \Pi^T]$ and $\bar{B}^T=[B^T\,0 ]$) can be eliminated by using a null-space projection procedure. To see how this can be achieved we note that, if $A$ has full row rank, we can always construct a matrix $Z\in \mathbb{R}^{N\times \tilde{N}}$ whose columns span the null-space of $\bar{A}$ (i.e., $\bar{A}Z\tilde{z}=0$ for any $\tilde{z}\in \mathbb{R}^{\tilde{N}}$) and with $\tilde{N}<N$. Similarly, we can construct a matrix $Y\in\mathbb{R}^{N \times(N-\tilde{N})}$ whose columns span the range-space of $\bar{A}^T$ (i.e., $Y\tilde{y}\in \textrm{Range}(\bar{A}^T )$ for any $\tilde{y}\in \mathbb{R}^{N-\tilde{N}}$). We can express any $z\in \mathbb{R}^N$ as $z=Z\tilde{z}+Y\tilde{y}$. We thus have that $\bar{A}z=\bar{A}Y\tilde{y}$ for all $\tilde{z}$ and thus $\tilde{y}=-(\bar{A}Y)^{-1}\bar{B}d$ and $z=-Y(AY)^{-1}\bar{B}d+Z\tilde{z}$ satisfies $\bar{A}z=-\bar{B}d$ for any $\tilde{z}$. With this, we can express the quadratic objective as $ \frac{1}{2}{z}^T Q {z} -{c}^T{z}$ as $ \frac{1}{2}\tilde{z}^T Z^TQ Z\tilde{z} -{c}^TZ\tilde{z}-(Y(AY)^{-1} \overline{B} d)^T Q Z \tilde{z}+\kappa$, where $\kappa$ is a constant. We thus obtain a QP of the same form as \eqref{general} but with matrix $Q\leftarrow Z^TQZ$, reduced cost $c\leftarrow Z^Tc+Z^T Q Y (AY)^{-1} \overline {B} d$, and variable vector $z\leftarrow \tilde{z}$. We highlight that this reduction procedure does not need to be applied in a practical implementation but we only use it to justify that the formulation \eqref{general} is general.  

We thus have that the variables $z$ in \eqref{general} are coupled implicitly via the matrix $Q$ and we seek to express this problem in the lifted form. We proceed to partition the set $\mathcal{N}$ into a set of partitions $\mathcal{K}=\{1,..,K\}$ to give the partition sets $\mathcal{N}_1,\mathcal{N}_2,\cdots,\mathcal{N}_K\subseteq\mathcal{N}$ satisfying $\mathcal{N}=\mathcal{N}_1\cup \mathcal{N}_2\cup \cdots \cup \mathcal{N}_K$ and $\mathcal{N}_k \cap \mathcal{N}_{k'}=\emptyset \quad \text{for all } k,k'\in \mathcal{K}\text{ and } k\neq k'$.  Coupling between variables arises when $Q_{ij}\neq0$ for $i\in \mathcal{N}_k$, $j\in \mathcal{N}_k'$, and $k\neq k'$. We perform lifting by defining index sets:
\begin{align}\label{partitioning2}
\underline{\mathcal{N}}_k:=\{j\in \mathcal{N}\setminus \mathcal{N}_k \mid \exists\, i\in \mathcal{N}_k \text{ s.t. } Q_{ij}\neq0\},\quad \overline{\mathcal{N}}_k:=\mathcal{N}_k \cup \underline{\mathcal{N}}_k
\end{align} 
The set $\underline{\mathcal{N}}_k$ includes all the coupled variables in the partition set $\mathcal{N}_k$ that are not in partition $k$. The set $\overline{\mathcal{N}}_k$ includes all variables in partition $\mathcal{N}_k$ and its coupled variables. These definitions allow us to express problem \eqref{general} as:
\begin{equation}\label{partition1}
\min_{{z}}\quad\frac{1}{2}\sum_{k\in\mathcal{K}} \sum_{i\in \mathcal{N}_k} \sum_{j\in \overline{\mathcal{N}}_k} Q_{ij}z(i) z(j)-\sum_{k\in\mathcal{K}} \sum_{i\in \mathcal{N}_k} c_i z(i).
\end{equation}
To induce lifting, we introduce a new set of variables $\{{z}_1,{z}_2,\cdots,{z}_K\}$ defined as:
\begin{equation}\label{partitioning3}
{z}_k:=\begin{bmatrix}
z(k_1)\\ \vdots \\ z(k_{N_k})
\end{bmatrix},\quad
\underline{{z}}_k=\begin{bmatrix}
z(k_{N_k+1})\\
\vdots\\
z(k_{N_k+\underline{N}_k})
\end{bmatrix},\quad
\overline{{z}}_k:=\begin{bmatrix}
{z}_k\\ \underline{ {z}}_k
\end{bmatrix}
\end{equation}
where $\mathcal{N}_k=\{k_1, k_2, \cdots, k_{N_k}\}$, $\underline{\mathcal{N}}_k=\{k_{N_k+1},k_{N_k+2}, \cdots,k_{N_k+\underline{N}_k}\}$, $N_k$ is the number of variables in partition $\mathcal{N}_k$, and $\underline{N}_k$ is the number of variables coupled to partition $k$. With this, we can express problem \eqref{partition1} in the following {\em lifted} form:
\begin{subequations} \label{partition2}
\begin{align}
\min_{{z}}\quad&\sum_{k\in\mathcal{K}}\frac{1}{2} {\overline{{z}}_k}^T\overline{Q}_k \overline{{z}}_k- {\overline{{c}}_k}^T\overline{{z}}_k.\\
\text{s.t.}\quad&{\Pi}_{kk}\overline{{z}}_k+\sum_{k'\in\mathcal{K}\setminus\{k\} } {\Pi}_{kk'}\overline{{z}}_{k'}=0,\quad k\in\mathcal{K}.
\end{align}
\end{subequations}
Here, $\overline{Q}_k \in \mathbb{R}^{(N_k+\underline{N}_k)\times(N_k+\underline{N}_k)}$ and $\overline{c}_k \in \mathbb{R}^{(N_k+\underline{N}_k)}$ are given by:
\begin{subequations}
\begin{align}
\label{def Q}
(\overline{Q}_k)_{ij}=
\begin{cases}
Q_{k_i k_j}, &\text{for }k_i,k_j\in \mathcal{N}_k\\
\frac{1}{2} Q_{k_i k_j},  &\text{for } k_i\in \mathcal{N}_k \text{ and } k_j\notin \mathcal{N}_k\\
\frac{1}{2} Q_{k_i k_j},  &\text{for } k_i\notin \mathcal{N}_k \text{ and } k_j\in \mathcal{N}_k\\
0, \quad &\text{otherwise}
\end{cases},\qquad (\overline{{c}}_k)_{i}=\begin{cases}
c_{k_i}, &\text{for }k_i\in \mathcal{N}_k\\
0, &\text{for } k_i \notin \mathcal{N}_k
\end{cases}
\end{align}
\end{subequations}
The coefficient matrices $\Pi_{kk} \in \mathbb{R}^{\underline{N}_k\times(N_k+\underline{N}_k)}$ are given by:
\begin{equation}
\label{def pi}
\Pi_{kk}=\begin{bmatrix}
{e_{N_k+1}}^T\\
\vdots\\
{e_{N_k+\underline{N}_k}}^T
\end{bmatrix},
\end{equation} 
where $e_i\in \mathbb{R}^{N_k+\underline{N}_k}$ are elementary column vectors. Note that for $Q_{ij} \neq 0$, $i\in \mathcal{N}_k$, $j \in \mathcal{N}_{k'}$, and $k\neq k'$, $Q_{ij}$ can be included in the objective function of either partition $k$, partition $k'$, or both. In the lifting scheme shown in \eqref{def Q}, we assume that these terms are equally divided and included in each partition. However, this approach is arbitrary, and other lifting schemes are possible. In other words, we can manipulate the lifting scheme to set the partition sets to satisfy either $j \in \underline{\mathcal{N}}_k$ or $j \notin \underline{\mathcal{N}}_k$.  Interestingly, one can show that the solution of the lifted problem is unique and is the same as that of  problem \eqref{general}. The proof of this assertion is intricate and will not be discussed here due to the lack of space.  To simplify the notation we express the lifted variables $\bar{z}_k$, matrices $\bar{Q}_k$, and cost vectors $\bar{c}_k$ in \eqref{partition2} simply as $z_k,Q_k,c_k$.

The primal-dual solution of \eqref{partition2} can be obtained by solving the KKT conditions:

%
\begin{footnotesize}
\begin{equation}\label{general big2}
\begin{bmatrix}
Q_1&\Pi_{11}^T&&\Pi_{21}^T&&\cdots&&\Pi_{K1}^T \\
\Pi_{11}&&\Pi_{12}&&\cdots&&\Pi_{1K}& \\
&\Pi_{12}^T&Q_2&\Pi_{22}^T&&&&\vdots\\
\Pi_{21}&&\Pi_{22}&&&&\vdots& \\
&\vdots&&&\ddots&&&\vdots \\
\vdots&&&&&\ddots&\vdots& \\
&\Pi_{1K}^T&&\cdots&&\cdots&Q_K&\Pi_{KK}^T \\
\Pi_{K1}&&\cdots&&\cdots&&\Pi_{KK}& \\
\end{bmatrix}
\begin{bmatrix}
z_1\\
\lambda_1\\
z_2\\
\lambda_2\\
\vdots\\
z_K\\
\lambda_K
\end{bmatrix}=
\begin{bmatrix}
c_1\\
{0}\\
c_2\\
{0}\\
\vdots\\
c_K\\
{0}
\end{bmatrix}
\end{equation}
\end{footnotesize}
By exploiting the structure of this system, we can derive a GS scheme of the form:
\begin{equation}\label{general iter}
\begin{bmatrix}
Q_k&\Pi_{kk}^T\\
\Pi_{kk}&
\end{bmatrix}
\begin{bmatrix}
z_k^{\ell+1}\\
\lambda_k^{\ell+1}
\end{bmatrix}=
\begin{bmatrix}
c_k\\
{0}
\end{bmatrix}-
\sum_{k'=1}^{k-1}
\begin{bmatrix}
0&\Pi_{k'k}^T\\
\Pi_{kk'}&0
\end{bmatrix}
\begin{bmatrix}
z_{k'}^{\ell+1}\\
\lambda_{k'}^{\ell+1}
\end{bmatrix}
-
\sum_{k'=k+1}^{K}
\begin{bmatrix}
0&\Pi_{k'k}^T\\
\Pi_{kk'}&0
\end{bmatrix}
\begin{bmatrix}
z_{k'}^{\ell}\\
\lambda_{k'}^{\ell}
\end{bmatrix}.
\end{equation}
Here, we have used a lexicographic coordination order. We note that the solution of the linear system \eqref{general iter} solves the optimization problem:
\begin{subequations}
\label{general opt iter}
\begin{align}
z_k^{\ell+1}=\mathop{\textrm{argmin}}_{z_k} &\quad\frac{1}{2}{z_k}^TQ_k z_k- {{{z}}}_k^T\left({{{c}_k}} -\sum^{k-1}_{k'=1}\Pi_{k'k}^T\lambda_{k'}^{\ell+1}-\sum^{K}_{k'=k+1}\Pi_{k'k}^T\lambda_{k'}^\ell\right)\\
\text{s.t.}&\quad{\Pi}_{kk}z_k+\sum_{k'=1}^{k-1} {\Pi}_{kk'}z_{k'}^{\ell+1}+\sum_{k'=k+1}^{K} {\Pi}_{kk'}z_{k'}^{\ell}=0 \quad (\lambda_k).
\end{align}
\end{subequations}
From this structure, we can  see how the primal and dual variables propagate forward and backward relative to the partition $k$, due to the inherent block triangular nature of the GS scheme.  We now seek to establish a condition that guarantees convergence of the GS scheme in this more general setting.  To do so, we define: 
\begin{equation}
\label{general define}
A_k:=
\begin{bmatrix}
Q_k&{\Pi_{kk}}^T\\
\Pi_{kk}&
\end{bmatrix},\quad
{x}_k^\ell:=
\begin{bmatrix}
{z}_k^{\ell}\\
{\lambda}_k^{\ell}
\end{bmatrix}, \quad
{b}_k:=
\begin{bmatrix}
c_k\\
{0}
\end{bmatrix},\quad
B_{kk'}:=
\begin{bmatrix}
&-{\Pi_{k'k}}^T\\
-\Pi_{kk'}&
\end{bmatrix}.
\end{equation}
By using \eqref{general define}, we can express \eqref{general iter} as:
\begin{equation}\label{general new}
A_k{x}_k^{\ell+1}={b}_k+\sum^{k-1}_{k'=1}B_{kk'}{x}_{k'}^{\ell+1}+\sum^{K}_{k'=k+1}B_{kk'}{x}_{k'}^\ell.
\end{equation}
This can be expressed in matrix form as:
\begin{gather}
\begin{bmatrix}
A_1&&&\\
-B_{21}&A_2&& \\
\vdots&\ddots&\ddots&\\
-B_{K1}&\cdots&-B_{KK-1}&A_K
\end{bmatrix}
\begin{bmatrix}
{x}_1^{\ell+1}\\
{x}_2^{\ell+1}\\
\vdots\\
{x}_K^{\ell+1}
\end{bmatrix}
=
\begin{bmatrix}
\quad&B_{12}&\cdots&B_{1K}\\
\quad&& \ddots&\vdots \\
\quad&&&B_{K-1K}\\
\quad
\end{bmatrix}
\begin{bmatrix}
{x}_1^{\ell}\\
{x}_2^{\ell}\\
\vdots\\
{x}_K^{\ell}
\end{bmatrix}
+
\begin{bmatrix}
{b}_1\\
{b}_2\\
\vdots\\
{b}_K
\end{bmatrix}.\label{before rearrangement}
\end{gather}


We can see that the partition matrices $A_k$ are nonsingular by inspecting the block structure of $Q_{k}$. In particular, by the definition of $Q_{k}$ in \eqref{def Q} (we denoted $Q_k$ as $\overline{Q}_k$ here) and the definition of $\Pi_{kk}$ in \eqref{def pi}, we can express $A_k$ as:
\begin{equation}
A_k=\begin{bmatrix}
\hat{Q}_k & \underline{Q}_k^T & \\
{\underline{Q}_k} & & I\\
& I &
\end{bmatrix}
\end{equation}
where each components of $\hat{Q}_k\in \mathbb{R}^{N_k \times N_k}$ and $\underline{Q}_{k}\in \mathbb{R}^{\underline{N}_k \times N_k}$ are defined in \eqref{def Q}. Since $Q$ is positive definite, $\hat{Q}_k$ is also positive definite. Noting that $\hat{Q}_k$ is nonsingular, we can see that the columns of $A_k$ are linearly independent and thus $A_k$ is nonsingular as well. This implies that the block lower triangular matrix on the left-hand side of \eqref{before rearrangement} is also nonsingular. To simplify notation we define: 
\vspace{-0.1in}
\begin{subequations}
\begin{align}
{w}^{\ell}&:=
\begin{bmatrix}
{x}_1^{\ell}\\
{x}_2^{\ell}\\
\vdots\\
{x}_K^{\ell}
\end{bmatrix},\quad {r}:=
\begin{bmatrix}
A_1&&&\\
-B_{21}&A_2&& \\
\vdots&\ddots&\ddots&\\
-B_{K1}&\cdots&-B_{KK-1}&A_K
\end{bmatrix}^{-1}
\begin{bmatrix}
{b}_1\\
{b}_2\\
\vdots\\
{b}_K
\end{bmatrix}
\label{general define S}\\
S&:=
\begin{bmatrix}
A_1&&&\\
-B_{21}&A_2&& \\
\vdots&\ddots&\ddots&\\
-B_{K1}&\cdots&-B_{KK-1}&A_K
\end{bmatrix}^{-1}
\begin{bmatrix}
\quad&B_{12}&\cdots&B_{1K}\\
\quad&& \ddots&\vdots \\
\quad&&&B_{K-1K}\\
\quad
\end{bmatrix}.
\end{align}
\end{subequations}
We express \eqref{before rearrangement} by using the compact form ${w}^{\ell+1}=S{w}^\ell+{r}$ or $w^{\ell+1}=S^\ell{w}^0+\left(I+S+\cdots+S^{\ell-1}\right){r}$. The solution of \eqref{general big2} satisfies ${{w}}=S{w}+{r}$. This implies that the solution of the lifted problem also solves the original problem \eqref{general}.  We now formally establish the following convergence result for the GS scheme. 
\begin{prop}
\label{general suf cond}
The GS scheme \eqref{general opt iter} converges to the solution of \eqref{general} if all the eigenvalues of the matrix:
\begin{scriptsize}
\begin{align*}
\begin{split}
\Sigma:=\begin{bmatrix}
Q_1&&&&\Pi_{11}^T&&& \\
&Q_2&&&\Pi_{12}^T&\Pi_{22}^T&& \\
&&\ddots&&\vdots&\ddots&\ddots& \\
&&&Q_K&\Pi_{1K}^T&\cdots&\Pi_{K-1K}^T&\Pi_{KK}^T \\
\Pi_{11}&&&&&&& \\
\Pi_{21}&\Pi_{22}&&&&&& \\
\vdots&\ddots&\ddots&&&&& \\
\Pi_{K1}&\cdots&\Pi_{KK-1}&\Pi_{KK}&&&& \\
\end{bmatrix}^{-1}\begin{bmatrix}
&&&&&-\Pi_{21}^T&\cdots&-\Pi_{K1}^T \\
&&&&&&\ddots&\vdots \\
&&&&&&&-\Pi_{KK-1}^T \\
&&&&&&& \\
\quad&-\Pi_{12}&\cdots&-\Pi_{1K}&&&& \\
&&\ddots&\vdots&&&& \\
&&&-\Pi_{K-1K}&&&& \\
&&&&&&& \\
\end{bmatrix}
\end{split}
\end{align*}
\end{scriptsize}
have magnitude less than one.
\end{prop}
\begin{proof}
Matrix $\Sigma$ is a permutation of matrix $S$. The permutation is a similarity transformation and thus $S$ has the same set of eigenvalues as $\Sigma$. Consequently, all the non-zero eigenvalues of $S$ have a magnitude of less than one. This implies that all the eigenvalues of $S^\ell$ decay exponentially as $\ell\rightarrow \infty$. Furthermore, this indicates that the series $I+S+S^2+\cdots$ converges to $(I-S)^{-1}$. Note that $I-S$ is invertible since all the eigenvalue of $S$ have magnitude less than one. Accordingly, the solution of the scheme \eqref{general opt iter} satisfies $\lim_{\ell \rightarrow\infty}{w}^\ell= (I-S)^{-1}{r}$. This convergence value $(I-S)^{-1}{r}$ satisfies equation  ${{w}}=S{w}+{r}$. Note that this is a unique solution to equation ${{w}}=S{w}+{r}$ and thus solves \eqref{partition2}. Problem \eqref{partition2} has same solution of problem \eqref{general}, and both problems have unique solutions. Thus, the solution of the GS scheme \eqref{general opt iter} converges to the solution of \eqref{general}.  $\Box$
\end{proof}

Convergence is achieved without any assumptions on the initial guess ${w}^0$. The error between the solution of problem \eqref{general} and the solution of the $\ell$-th coordination update of the GS scheme \eqref{general iter} (i.e., $(I-S)^{-1}{r}-{w}^\ell$) is given by ${\epsilon}^\ell=S^\ell\left((I-S)^{-1}{r}-{w}^0\right)$. We will call $\Vert\epsilon^\ell\Vert$ the {\em error} of the $\ell$-th coordination step. If we express the initial error term as $\Vert\epsilon^0\Vert$, we can write $\Vert\epsilon^\ell\Vert=\Vert S^\ell\epsilon^0\Vert$. Moreover, if we define $\rho(S)=\lambda_{max}(S)$ then, for any $\delta>0$, there exists $\kappa>0$ such that the bound $\Vert S^\ell \Vert \leq \kappa \vert \rho(S) +\delta \vert^\ell$ holds for all $\ell$ and where $\Vert S \Vert$ is a matrix norm of $S$. Using this inequality we can establish that $\Vert\epsilon^\ell\Vert \leq \kappa \vert \rho(S)+\delta \vert^\ell \Vert \epsilon^0 \Vert$. Since we can choose $\delta$ to be arbitrarily small, we can see that the decaying rate of the error is $O\left(\rho(S)^\ell \right)$. Although the error decays exponentially we note that, if $\rho(S)$ is close to one, convergence can be slow and thus having a good initial guess ${w}^0$ is essential.  We also note that $\rho(S)$ is {\em tightly related to the topology of the coupling} between partitions, indicating that the partition structure contributes to the convergence rate. 

We highlight that the GS concepts discussed here focus on problems with no inequality constraints. In practice,  however, GS schemes can also be applied to such problems by using projected GS schemes \cite{zavalamultigrid,borzi2005multigrid,zavala2010real}. 

\vspace{-0.2in}\subsection{Coordination Orders}
In the lexicographic coordination order proposed in \eqref{general iter}, we use the update sequence $k=1,2,\cdots,K$. We note that the order affects the structure of the matrix $\Sigma$ and thus convergence can be affected as well. To see this, consider a new order given by the sequence $\sigma(1),\sigma(2),\cdots,\sigma(K)$  where $\sigma:\{1,2,\cdots,K\}\rightarrow\{1,2,\cdots,K\}$ is a bijective mapping. We can use this mapping to rearrange the partition variables as:
\begin{equation}
[z_1,
\lambda_1,
z_2,
\lambda_2,
\dots
z_K,
\lambda_K]
\rightarrow
[z_{\sigma(1)},
\lambda_{\sigma(1)},
z_{\sigma(2)},
\lambda_{\sigma(2)},
\dots
z_{\sigma(K)},
\lambda_{\sigma(K)}].
\end{equation}
This gives the reordered matrix: 
\begin{gather}\label{rearranged S}
S=
\begin{bmatrix}
A_{\sigma(1)}&&&\\
-B_{\sigma(2)\sigma(1)}&A_{\sigma(2)}&& \\
\vdots&\ddots&\ddots&\\
-B_{\sigma(K)\sigma(1)}&\cdots&-B_{\sigma(K)\sigma(K-1)}&A_{\sigma(K)}
\end{bmatrix}^{-1}
\begin{bmatrix}
\quad&B_{\sigma(1)\sigma(2)}&\cdots&B_{\sigma(1)\sigma(K)}\\
\quad&& \ddots&\vdots \\
\quad&&&B_{\sigma(K-1)\sigma(K)}\\
\quad
\end{bmatrix}
\end{gather}
Importantly, the {\em change in update order is not necessarily a similarity transformation} of matrix $S$ and thus the eigenvalues will be altered.  The GS scheme converges as long as eigenvalues of the reordered matrix $S$ have magnitude less than one but the convergence rate will be affected. Interestingly, GS schemes are highly flexible and allow for a large number of update orders. For instance, in some cases one can derive ordering sequences that enable parallelization. As an example, the 1-D spatial problem has a special structure with $B_{kk'}$=0 for any $(k,k')$ such that $\vert k-k' \vert \geq 2$. The GS scheme becomes:
\begin{align}\label{order change iter}
A_k{x}_k^{\ell+1}&={b}_k+B_{kk-1}{x}_{k-1}^{\ell+1}+B_{kk+1}{x}_{k+1}^\ell
\end{align}

Instead, we consider the following ordering $\sigma(i)=2i-1$ for $1\leq i \leq \frac{K}{2}$ and $\sigma(i)=2i-K$ for $\frac{K}{2}+1\leq i \leq K$. Here we assume that the number of partitions is even. This is a called a {\em red-black ordering} and is widely popular in the solution of PDEs. By changing the index using $\sigma(\cdot)$, we can express \eqref{order change iter} as: 

\begin{footnotesize}
\begin{subequations}\label{order change iter sigma a}
\begin{align}
A_{\sigma(i)}{x}_{\sigma(i)}^{\ell+1}&={b}_{\sigma(i)}+B_{{\sigma(i)}{\sigma(i+\frac{K}{2})}}{x}_{\sigma(i+\frac{K}{2})}^\ell,\quad i=1\\
A_{\sigma(i)}{x}_{\sigma(i)}^{\ell+1}&={b}_{\sigma(i)}+B_{{\sigma(i+\frac{K}{2}-1)}{\sigma(i+\frac{K}{2}-1)}}{x}_{\sigma(i+\frac{K}{2}-1)}^{\ell}+B_{{\sigma(i)}{\sigma(i+\frac{K}{2})}}{x}_{\sigma(i+\frac{K}{2})}^\ell,\quad 2\leq i \leq \frac{K}{2}
\end{align}
\end{subequations}
and
\begin{subequations}\label{order change iter sigma b}
\begin{align}
A_{\sigma(i)}{x}_{\sigma(i)}^{\ell+1}&={b}_{\sigma(i)}+B_{{\sigma(i)}{\sigma(i-\frac{K}{2})}}{x}_{\sigma(i-\frac{K}{2})}^{\ell+1}+B_{{\sigma(i)}{\sigma(i-\frac{K}{2}+1)}}{x}_{\sigma(i-\frac{K}{2}+1)}^{\ell+1},\quad \frac{K}{2}+1\leq i \leq K-1.\\
A_{\sigma(i)}{x}_{\sigma(i)}^{\ell+1}&={b}_{\sigma(i)}+B_{{\sigma(i)}{\sigma(i-\frac{K}{2})}}{x}_{\sigma(i-\frac{K}{2})}^{\ell+1},\quad i =K.
\end{align}
\end{subequations}
\end{footnotesize}
We can thus see that the solution of \eqref{order change iter sigma a} can proceed {\em independently} for any $1\leq i \leq \frac{K}{2}$. This is because these partitions only depend on the solutions of \eqref{order change iter sigma b} but not on the solutions of \eqref{order change iter sigma a}. Likewise, solving \eqref{order change iter sigma b} can be done independently for any $\frac{K}{2}+1\leq i \leq K$. Red-black ordering thus enables parallelism.  Different coordination orders for 1-D and 2-D meshes are presented in Figures \ref{1D order} and \ref{2D order}.  

\begin{figure}[!htb]
\begin{center}
\includegraphics[width=2in]{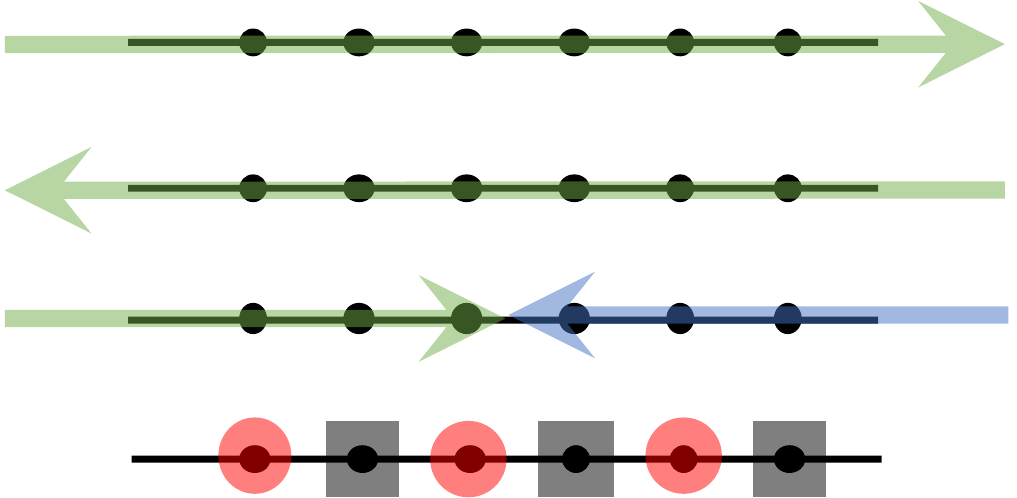}\caption{Sketch of 1-D ordering methods.}\label{ordering scheme temporal}\label{1D order} \vspace{0.1in}
\includegraphics[width=4.5in]{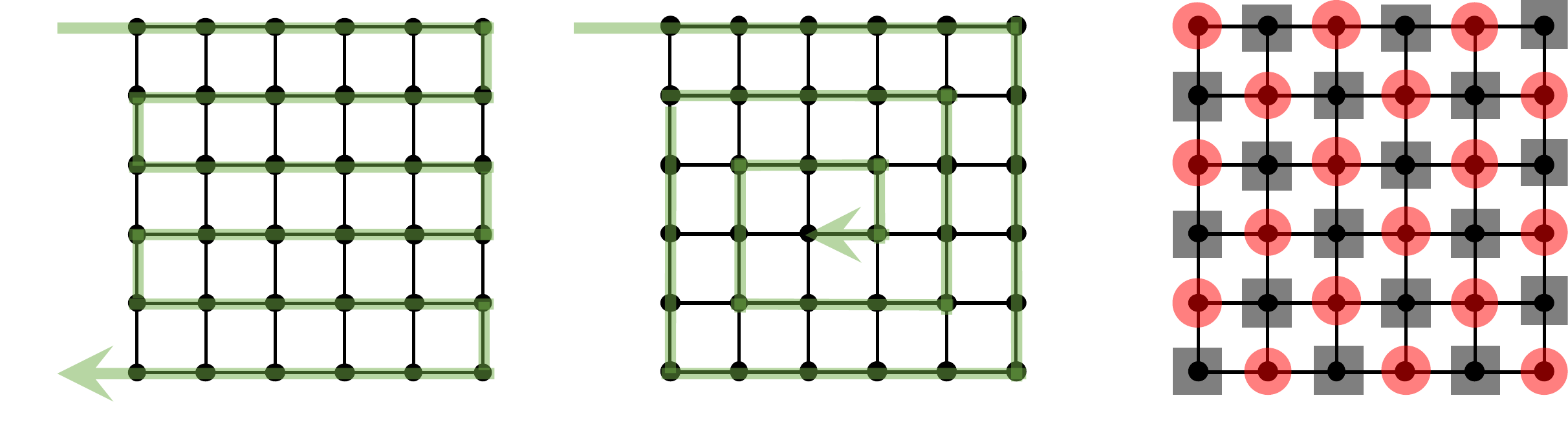}\caption{Sketch of 2-D ordering methods.}\label{ordering scheme spatial}\label{2D order}
\end{center}
\end{figure}

\vspace{-0.2in}\subsection{Coarsening}

Low-complexity coarse versions of the full-resolution problem \eqref{multi-scale control} can be solved to obtain an initial guess for the GS scheme and with this accelerate coordination.  To illustrate how this can be done, we use the following representation of \eqref{multi-scale control}: 
\begin{subequations}\label{coarsening original}
\begin{align}
\min_{{z}}\quad & \frac{1}{2}{z}^T Q {z}-{c}^T{z} \\
\text{s.t.}\quad & \underbrace{\begin{bmatrix} A \\ \Pi \end{bmatrix}}_{\bar{A}}{z}+\underbrace{\begin{bmatrix}B \\ 0 \end{bmatrix}}_{\bar{B}} d=0 \quad (\nu, {\lambda})
\label{eq:cons}
\end{align}
\end{subequations}
Our goal is to obtain a substantial reduction in the dimension of this problem by introducing a {\em mapping} from a coarse variable space to the original space. This is represented by the linear mapping ${z}=T\tilde{{z}}$  where  $\tilde{{z}}\in \mathbb{R}^{N_c\cdot n_z}$ is the coarse variable and we assume that the mapping $T\in \mathbb{R}^{N\cdot n_z\times N_c\cdot n_z}$ (called a restriction operator) has full column rank and $N_c<N$. We can thus pose the low-dimensional coarse problem: 
\begin{subequations}\label{coarse problem}
\begin{align}
\min_{{z}}\quad & \frac{1}{2}{\tilde{z}}^T T^TQT {\tilde{z}}-{c}^TT{\tilde{z}} \\
\text{s.t.}\quad & U^T\bar{A} T{\tilde{z}}+U^T\bar{B} d=0 \quad (\tilde{\nu},\tilde{\lambda})
\end{align}
\end{subequations}
A key issue that arises in coarsening is that the columns of matrix $\bar{A} T$ {\em do not necessarily span the entire range space of $\bar{B}$} (i.e., we might not be able to find a coarse variable $\tilde{z}$ that satisfies $\bar{A} T{\tilde{z}}+\bar{B}d=0$). Consequently,  we introduce a {\em constraint aggregation} matrix $U$ that has full column rank to ensure that $U^T\bar{A}T$ spans the range space of $U^T\bar{B}$. With this, we can ensure that the feasible set of  \eqref{coarse problem} is non-empty. 

After solving the coarse problem \eqref{coarse problem}, we can project the primal-dual solution from the coarse space to the original space. We note that the dimension of the dual  space is also reduced because we performed constraint aggregation. The projection for the primal solution can be done by using ${z}=T\tilde{{z}}$ while the projection for the dual solution can be obtained as $(\nu,\lambda)=U(\tilde{\nu},{\tilde{\lambda}})$.  The derivation of coarsened representations is {\em application-dependent} and often requires domain-specific knowledge. In particular, coarsening can also be performed by using reduced order modeling techniques such as proper orthogonal decompositions \cite{antoulas} or coherency-based network aggregation schemes \cite{kokotovic1982coherency}. In the following sections we demonstrate how to derive coarse representations in certain settings.

\vspace{-0.2in}\subsection{Multi-Grid Schemes and Hierarchical Coordination}

Multi-grid serves as a bridge between a fully centralized and a fully decentralized coordination schemes. In particular, a fully centralized scheme would aim to find a solution of the full-resolution problem \eqref{lifted problem} by gathering all the information in a single processing unit. A fully decentralized scheme such as GS, on the other hand, would proceed by finding solutions to subproblems \eqref{iterative scheme} over each partition and information would only be shared between the connected partitions through the coordination variables.  A drawback of a decentralized scheme is that a potentially large number of coordination steps might be needed to reach a solution for the full-resolution problem, particularly when many partitions are present.   

In a multi-grid scheme, we seek to aid the decentralized scheme by using information from a low-resolution central problem that oversees the entire domain. In our context, the information is in the form of states and dual variables defined over the partition interfaces (i.e., the coupling variables and constraints). The key idea of this hierarchical arrangement is that the coarse central scheme can capture effects occurring at low global frequencies while the agents in the decentralized schemes can handle effects occurring at high local frequencies. As can be seen, multi-grid provides a framework to  design hierarchical control architectures by leveraging existing and powerful reduced order modeling techniques such as coherency-based aggregation and decentralized control schemes.   

Multi-grid is a widely studied computational paradigm. The framework proposed here presents basic elements of this paradigm but diverse extensions are possible \cite{borzi,brandtalgebraic}. For instance, the scheme proposed here involves only a coarse and a fine resolution level but one can create a multi-level schemes that transfer information between multiple scales recursively by using meshes of diverse resolution. This can allow us to cover a wider range of frequencies present in the system. In the following section we illustrate how sequential coarsening can be beneficial.  

From an {\em electricity markets} perspective, we highlight that multi-grid schemes provide a framework to coordinate transactions at multiple spatial and temporal scales. To see this, consider the case of spatial coordination of electricity markets. Under such setting, we can interpret each partition of the spatial domain as a market player (e.g., a microgrid). The market players have internal resources (e.g., distributed energy resources) that they manage to satisfy their internal load. The players, however, can also transact energy with other players in order to improve their economic performance. The proposed GS scheme provides a mechanism to handle intra-partition decision-making (by solving the partition subproblems) and inter-partition transactions by exchanging state (voltages) and dual information (nodal prices). If transaction information is exchanged multiple times (corresponding to multiple GS iterates), the GS scheme will converge to an equilibrium point corresponding to the solution of the centralized economic maximization problem (e.g., the social welfare problem). This is a useful property of decentralized coordination because centralization of information and decision-making is often impractical. If the players only exchange information once (or a handful of times) they might not reach an optimal equilibrium and an inefficiency will be introduced.  Moreover, when a disturbance affects the system, many GS iterations might be needed to reach the new optimal equilibrium. This is where hierarchical optimization becomes beneficial, because one can solve and aggregated spatial representation of the system (in which each partition is treated as a node) to compute approximate dual variables and states at the interfaces of the partitions. This approximation can be used to aid the convergence of the decentralized GS scheme (by conveying {\em global} spatial information to {\em local} market players). One can think of the coarse high-level problem as a system operator (supervisor) problem (e.g., at the distribution level). The operator might, at the same time, need to coordinate with other system operators (each of which oversees its own set of market players). These system operators can at the same time be aggregated into a higher level which would represent, for instance, a transmission or regional operator. We can thus see that hierarchical multi-grid schemes enable scalable coordination of potentially large number of market players over large geographical regions.  The hierarchical multi-grid scheme can also be applied to handle multiple timescales of a single market player that might need to manage, for instance, assets with different dynamic characteristics.  

\vspace{-0.2in}\section{Case Studies}\label{sec:cases}
We now present numerical case studies to demonstrate the concepts in the context of temporal and spatial management of energy systems. We use a multi-scale (in time) optimization problem with features of an storage management problem and a multi-scale (in space) optimization problem that considers power flow dispatch over a network. 

\vspace{-0.2in}\subsection{Multi-Scale Temporal Control}
We use a multi-grid scheme to solve the following temporal planning problem $\mathcal{P}$:
\begin{subequations}\label{temporal problem}
\begin{align}
\min_{{x},{u}}\quad & \sum_{i\in\mathcal{N}}(x(i)^2+{u(i)}^2) \\
\label{dynamic equation}\text{s.t.}\quad  & x(i+1)=x(i)+\delta (u(i+1)+d(i+1)),\quad i\in\mathcal{N} \\
&\label{boundary constriants}x(0)=0.
\end{align}
\end{subequations}
This problem has a state, a control, and a disturbance defined over $N=M\cdot K$ time points contained in the set $\mathcal{N}$. The state and control are grouped into the decision variable $z(i)=(x(i),u(i))$. The distance between mesh points is given by $\delta$. The structure of this problem resembles that of an inventory (storage) problem in which the disturbance $d(i)$ is a load and $u(i)$ is a charge/discharge flow from the storage. In these types of problems, the load might have frequencies covering multiple timescales (e.g., seasonal, daily, and down to seconds). Consequently, the time mesh $\delta$ has to be rather fine to capture all the frequencies. Moreover, the planning horizon (the time domain $N\cdot \delta$) might need to be long so as to capture the low frequency components in the load.   We partition the problem into $K$ partitions, each containing $M$ points. The set of inner points in the partition is defined as $\mathcal{M}$. The optimization problem over a partition $k$ solved in the GS scheme is given by:
\begin{subequations}\label{temporal GS problem}
\begin{align}
\min_{{x}_k,u_k}\quad & \sum_{i\in\mathcal{M}}\left(x_k(i)^2+{u_k(i)}^2\right)+x_k(M)\lambda_{k+1}^{\ell} \\
\text{s.t.}\quad  & x_k(i+1)=x_k(i)+\delta (u_k(i+1)+d_k(i+1)),\quad i\in\mathcal{M}\\
&x_k(0)=x_{k-1}^{\ell+1}(M) \quad (\lambda_k).
\end{align}
\end{subequations}
In our numerical experiments, we set $K=10$ and $M=100$ to give $N=1,000$ points. The time mesh points were set $t(i)=i\cdot \delta$ with $\delta=0.1$.  We use a disturbance signal composed of a low and a high frequency $d(i)={4\sin \left(\frac{4\pi i}{N}\right)}+{\sin\left(\frac{24\pi i}{N}\right)},\, i\in\mathcal{N}$. The disturbance signal and its frequency components are shown in Figure \ref{temporal disturbance figure}. 

\begin{figure}[!htb]
\begin{center}
\includegraphics[width=4in]{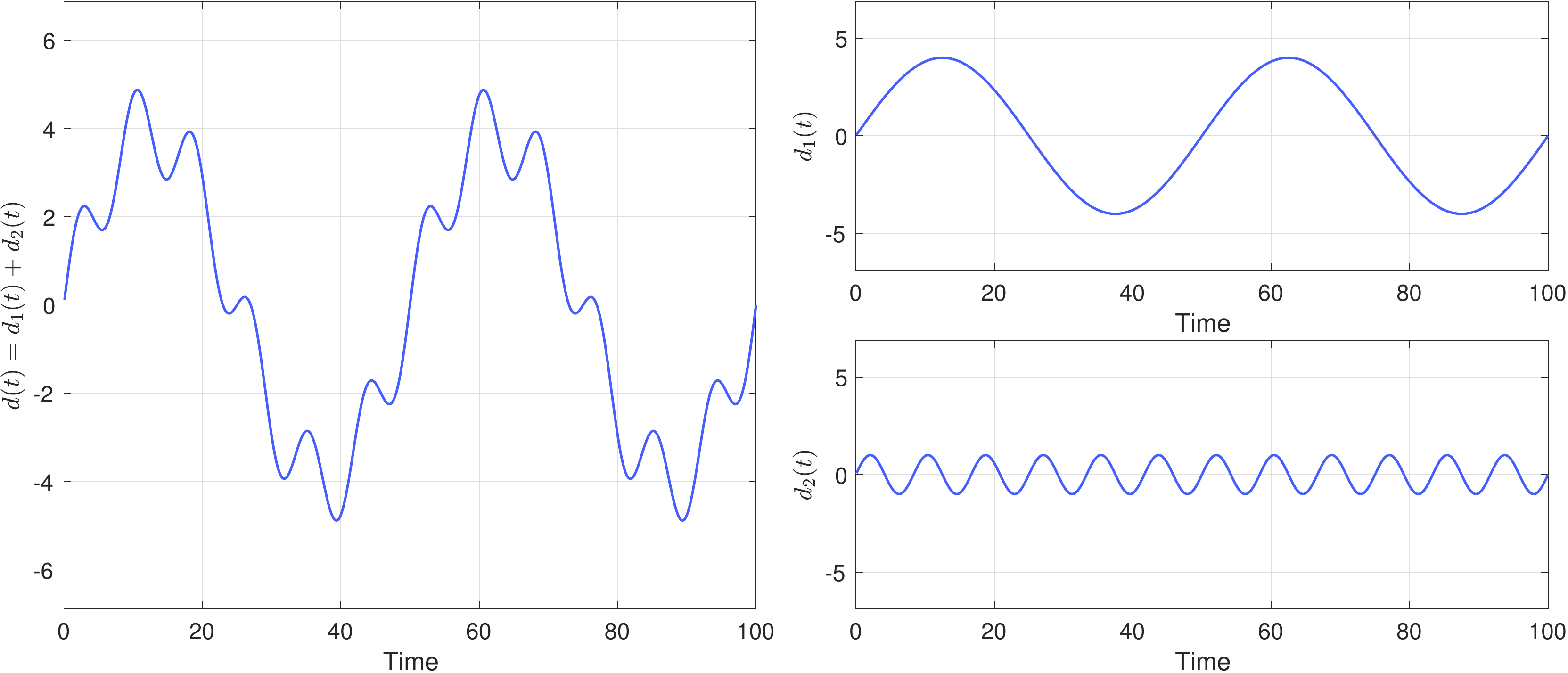}\caption{Disturbance profile (left) and frequency components (right) for temporal problem.}\label{temporal disturbance figure}
\end{center}
\end{figure}

We solve a coarse problem to create a hierarchical structure that aids the GS scheme (which operates on the partitions at high resolution). To perform coarsening, we aggregate $M_c=4$ internal grid points (i.e., collapse 25 points into one coarse point). The set of inner coarse points is $\mathcal{M}_c$. The projection is ${x}_k=T_x\tilde{{x}}_k$ and ${u}_k=T_u\tilde{{u}}_k$ with:

\begin{subequations}\label{coarsening scheme temporal}
\begin{gather}
T_x=T_u=\begin{bmatrix}
1\\
&1 \\
&\vdots \\
&1 \\
&&\ddots\\
&&&1\\
&&&\vdots\\
&&&1
\end{bmatrix}
\end{gather}
\end{subequations}
The projection matrices $T_x$ and $T_u$ have full column rank. The fine to coarse restriction can also be expressed as:
\begin{align}
\tilde{x}_k(\tilde{i})=x_k(i),\quad
\tilde{u}_k(\tilde{i})=u_k(i),\quad
\text{if }\quad \tilde{i}=\lfloor{\frac{i-1}{M/M_c}}\rfloor+1
\end{align}
where $\floor{\cdot}$ is a round-down operator. We denote $\varphi(i)=\lfloor{\frac{i-1}{M/M_c}}\rfloor+1$. The coarse problem can then be stated in terms of the coarse variables as follows:
\begin{subequations}\label{temporal coarse problem}
\begin{align}
\min_{{\tilde{x}}_{k},\tilde{u}_k}\quad & \sum_{j\in\mathcal{M}_c}\left(\tilde{x}_{k}(j)^2+{\tilde{u}_k(j)}^2\right)+\tilde{x}_k(M_c){\lambda}_{k+1}^\ell \\
\label{dmc}\text{s.t.}\quad  & \tilde{x}_k(j+1)=\tilde{x}_k(j)+\frac{M}{M_c}\delta (\tilde{u}_k(j+1)+\tilde{d}_k(j+1)),\quad j\in \mathcal{M}_c\\
&\tilde{x}_k(0)=\tilde{x}_{k-1}^{\ell+1}(M_c)\quad (\lambda_k),
\end{align}
\end{subequations}
where $\tilde{d}_k(j)=\frac{M_c}{M}\sum_{i\in\varphi^{-1}(\tilde{j})} d_k(i)$ is the coarsened disturbance signal.  The dynamic equations are defined over a smaller dimensional space (defined over $\mathcal{M}_c$) which results from aggregating the dynamic equations in the full resolution problem (defined over $\mathcal{M}$).  We solve the full resolution problem \eqref{temporal problem} and compare its solution against that of the pure GS scheme, the one of the coarse low resolution problem, and the one of the hierarchical scheme that solves the coarse problem to coordinate the GS scheme.  We also solve the coarse problem by using a GS scheme.  Figure \ref{temporal solution figure} shows the results. We note that the solution to the coarse problem \eqref{temporal GS problem} captures the general long-term trend of the solution but misses the high frequencies. The GS scheme refines this solution and converges in around 30 coordination steps. By comparing Figure \ref{temporal solution figure} (top-right) and (bottom-right), we can see that initializing GS scheme with the coarse solution significantly reduces the initial error, and demonstrates the benefit of the hierarchical scheme.

\begin{figure}[!htb]
\begin{center}
\includegraphics[width=4.0in]{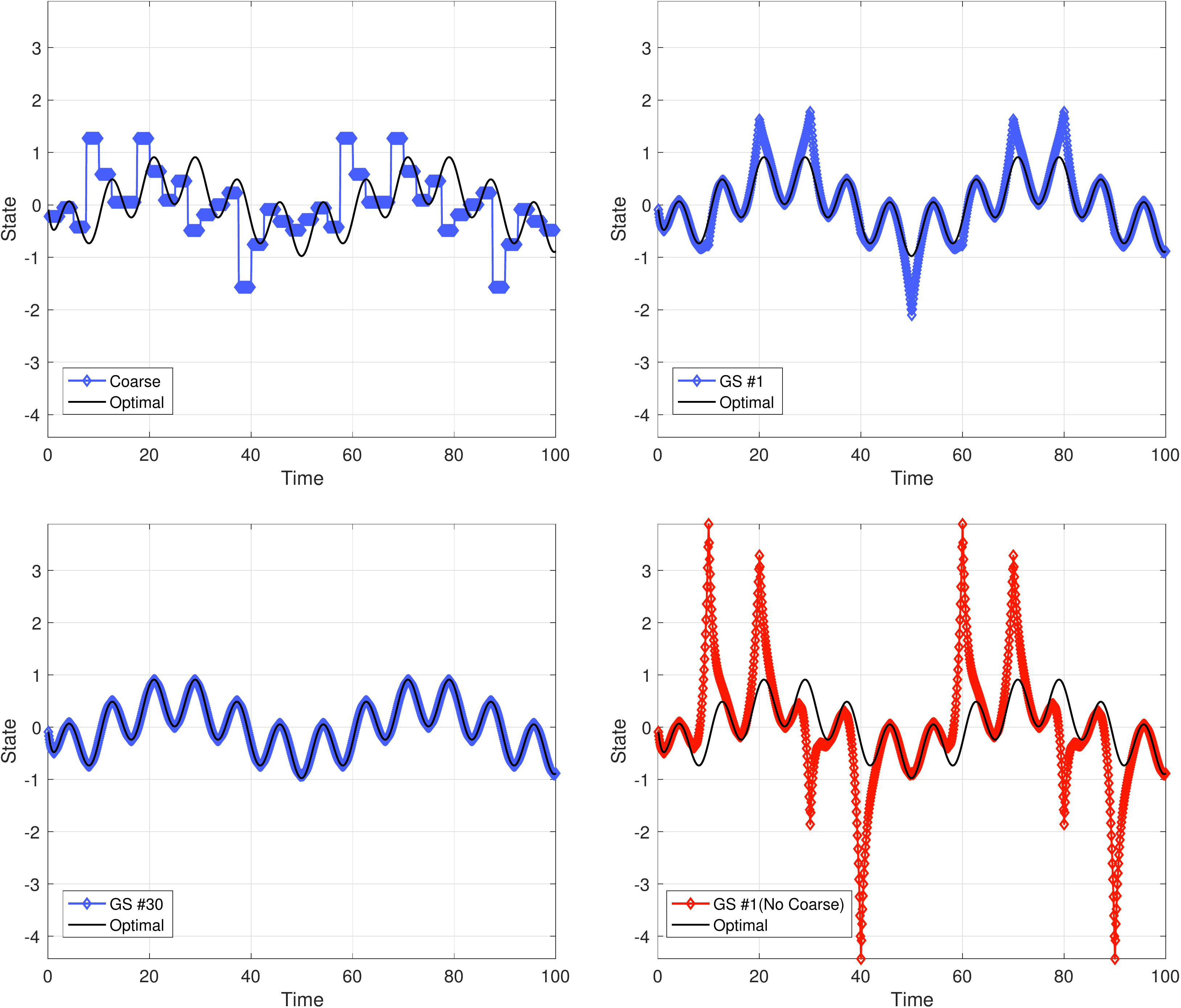}\caption{(Top-left) solution of coarse problem, (top-right) solution of first GS step with coarsening, (bottom-left) solution of  30th GS step with coarsening, (bottom-right) solution of first GS step without coarsening.}\label{temporal solution figure}
\end{center}
\end{figure}

\vspace{-0.5in}\subsection{Multi-Scale Spatial Control}

Useful insights on how to use multi-grid schemes to create hierarchical network control structures result from {\em interpreting network flows as diffusive processes}. To illustrate this, we consider a network with nodes defined on a rectangular mesh. A node $(i,j)$ in the network exchanges flows with its four neighboring nodes $(i,j+1),(i,j-1),(i+1,j),(i-1,j)$ (this is called a stencil). The flows $f({i,j})$ are a function of the node potentials  $p({i,j})$ and given by:
\begin{subequations}
\begin{align}
f({i,j;i,j+1})&=D(p({i,j})-p({i,j+1}))\\
f({i,j;i,j-1})&=D(p({i,j})-p({i,j-1}))\\
f({i,j;i+1,j})&=D(p({i,j})-p({i+1,j}))\\
f({i,j;i-1,j})&=D(p({i,j})-p({i-1,j})).
\end{align}
\end{subequations}
Here, $D\in \mathbb{R}$ is the diffusion constant (i.e., the flow resistance) of the link connecting the nodes. At each node $(i,j)$ we have a load $d(i,j)$ and a source $u(i,j)$ that is used to counteract (balance) the load. This gives the flow balance conservation equation:
\begin{align*}
f({i,j;i,j+1})+f({i,j;i,j-1})+f({i,j;i+1,j})+f({i,j;i-1,j})=u(i,j)+d(i,j).
\end{align*}
This can also be written in terms of the potentials as:
\begin{align*}
D\left(4\cdot p(i,j)-p(i-1,j)-p(i+1,j)-p(i,j-1)-p(i,j+1)\right)=u(i,j)+d(i,j).
\end{align*}

\begin{figure}[!htb]
\begin{center}
\includegraphics[width=4.5in]{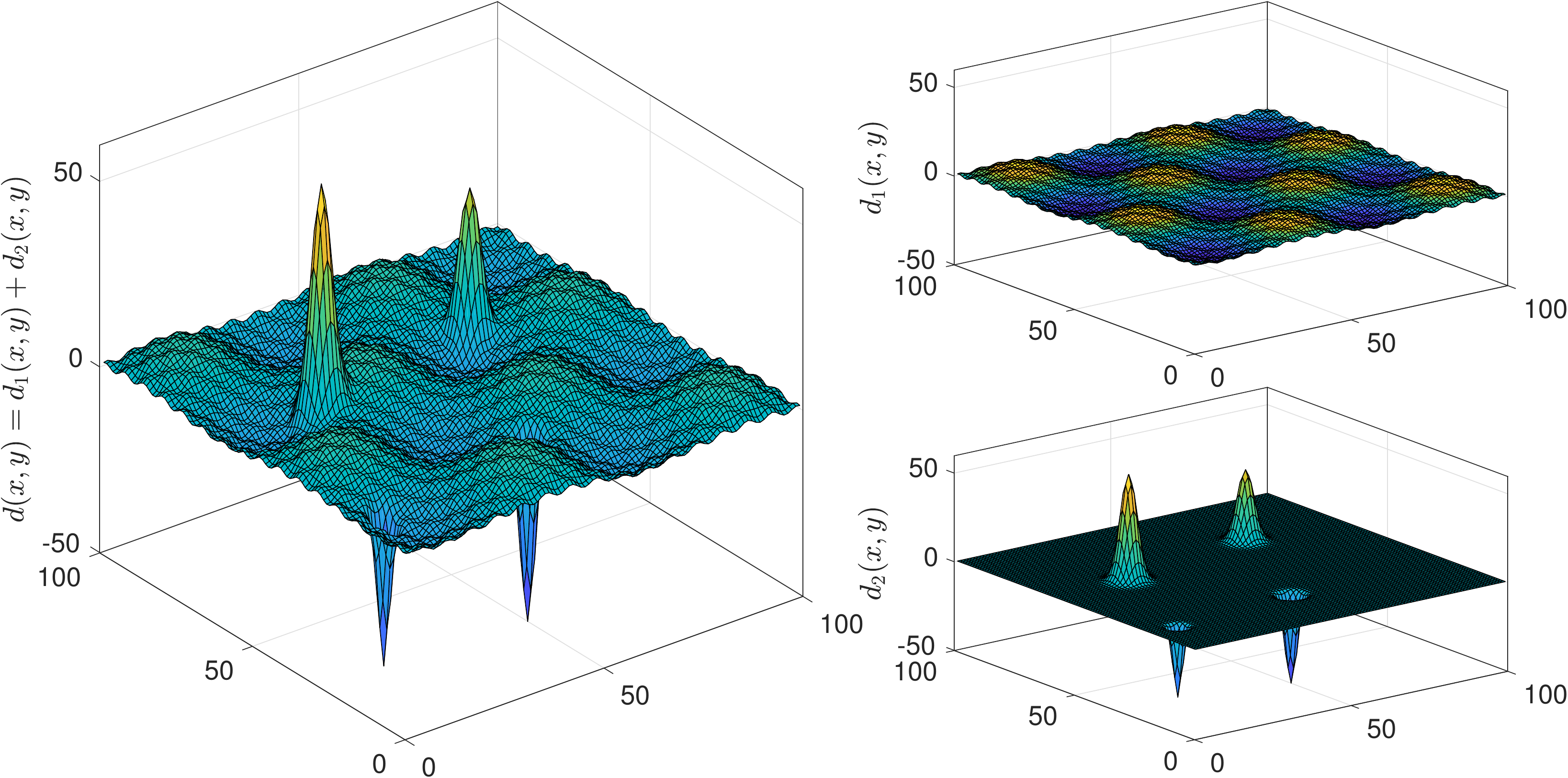}\caption{Disturbance field (left) and its components (right) for spatial optimization problem.}\label{spatial disturbance figure}
\end{center}
\end{figure}

We assume that we have fixed 2-D spatial domain $\Omega:=[0,X]\times [0,Y]$ that is discretized using $M\cdot P$ nodes in each direction. The sets $\mathcal{N}^x=\mathcal{N}^y:=\{1,2,\cdots,P\cdot M\}$ are the sets of points in each direction. The set of total mesh points is $\mathcal{N}=\mathcal{N}^x\times\mathcal{N}^y$ and thus $N=(P\cdot M)\cdot (P\cdot M)$. As the number of nodes increases, the node potentials form a continuum described by the 2-D diffusion equation:
\begin{align}
D\left(\frac{\partial^2 p(x,y)}{\partial x^2}+\frac{\partial^2 p(x,y)}{\partial y^2}\right)=u(x,y)+d(x,y),\quad (x,y)\in \Omega.  
\end{align}
Using this analogy, we consider the following full-space problem:
\begin{subequations}\label{spatial problem}
\begin{align}
\min_{{p},{u}} \quad&\sum_{(i,j)\in \mathcal{N}}\left(p(i,j)^2 +u(i,j)^2\right) \\
\text{s.t.}\quad & D\left(4\cdot p(i,j)-p(i-1,j)-p(i+1,j)-p(i,j-1)-p(i,j+1)\right)\nonumber \\
 &\qquad\qquad =u(i,j)+d(i,j),\quad (i,j)\in\mathcal{N}\\
& p(0,j)=0,\quad j\in\mathcal{N}^y\\
& p(M\cdot P+1,j)=0,\quad j\in\mathcal{N}^y\\
& p(i,0)=0,\quad i\in\mathcal{N}^x\\
& p(i,M\cdot P+1)=0,\quad i\in\mathcal{N}^x
\end{align}
\end{subequations}
The goal of the optimization problem is, given the loads $d(i,j)$, to control the potentials in the network nodes $p(i,j)$ by using the sources $u(i,j)$. The decision variables at every node are $z(i,j)=(p(i,j),u(i,j))$.  The presence of multiple frequencies in the {\em 2-D disturbance load field} $d(i,j)$ might require us to consider fine meshes, making the optimization problem intractable. One can think of the {\em disturbance field} as spatial variations of electrical loads observed over a geographical region.  If the loads have high frequency spatial variations, it would imply that we need high control resolution (i.e., we need sources at every node in the network to achieve tight control).  This can be achieved, for instance, by installing distributed energy resources (DERs). Moreover, if the load has low-frequency variations, it would imply that the DERs would have to cooperate to counteract global variations. 

In our experiments, the size of mesh was set $P=10$ and $M=10$, which results in $N=10,000$ mesh points. To address this complexity, we partition the 2-D domain into $K=P\cdot P$ partitions each with $M\cdot M$ points and we label each element in the partition as $k=(n,m)\in \mathcal{K}$. We can think of each partition $k\in \mathcal{K}$ as a region of the network. We define inner index sets by: $\mathcal{M}^x=\mathcal{M}^y:=\{1,2,\cdots,M\}$ and $\mathcal{M}:=\mathcal{M}^x \times \mathcal{M}^y$. The GS scheme for partition $(n,m)$ is given by:
\begin{subequations}\label{spatial GS problem}
\begin{align}
\min_{p_{n,m},{u}_{n,m}} \quad&\sum_{(i,j)\in \mathcal{M}}\left(p_{n,m}(i,j)^2 +u_{n,m}(i,j)^2\right) \\
\nonumber&
+\sum_{j\in\mathcal{M}^y} p_{n,m}(1,j)\lambda_{n-1,m}^{\ell+1}(M+1,j)
+\sum_{j\in\mathcal{M}^y} p_{n,m}(M,j)\lambda_{n+1,m}^{\ell}(0,j)\\
\nonumber&
+\sum_{i\in\mathcal{M}^x} p_{n,m}(i,1)\lambda_{n,m-1}^{\ell+1}(i,M+1)
+\sum_{i\in\mathcal{M}^x} p_{n,m}(i,M)\lambda_{n,m+1}^{\ell}(i,0)
\\
\text{s.t.}\quad & D\left(4\cdot p_{n,m} (i,j)-p_{n,m} (i-1,j)-p_{n,m} (i+1,j)-p_{n,m} (i,j-1)-p_{n,m} (i,j+1)\right)\nonumber\\
&\qquad\qquad=u_{n,m}(i,j)+d_{n,m}(i,j),\quad(i,j)\in \mathcal{M}\\
& p_{n,m}(0,j)=p_{n-1,m}^{\ell+1}(M,j),\quad\quad\;\; (\lambda_{n,m}(0,j)) \label{eq:eq1}\\
& p_{n,m}(M+1,j)=p_{n+1,m}^{\ell}(1,j),\quad (\lambda_{n,m}(M+1,j)) \\
& p_{n,m}(i,0)=p_{n,m-1}^{\ell+1}(i,M),\quad\qquad (\lambda_{n,m}(i,0)) \\
& p_{n,m}(i,M+1)=p_{n,m+1}^{\ell}(i,1), \quad\;\; (\lambda_{n,m}(i,M+1)),\quad \label{eq:eq4} \end{align}
\end{subequations}
The constraint indices for the constraints \eqref{eq:eq1}-\eqref{eq:eq4} run over $j\in\mathcal{M}^y$ and $ i\in\mathcal{M}^x$.

To perform coarsening, a mesh of $(M/M_c)\cdot(M/M_c)$ points is collapsed into a single coarse point and the mapping from the coarse space to the original space is:
\begin{equation}
\tilde{p}_{n,m}(\tilde{i},\tilde{j})=p_{n,m}(i,j)\quad\text{if  }\quad \tilde{i}=\lfloor\frac{i-1}{M/M_c}\rfloor+1,\quad \tilde{j}=\lfloor\frac{j-1}{M/M_c}\rfloor+1. 
\end{equation}
As with the temporal case, we also perform aggregation of the constraints in the partition to obtain a coarse representations. In our experiments, we used $M_c=2$ as default. The disturbance field is given by a linear combination of a 2-D sinusoidal and of a Gaussian function. The shape of the load field illustrated in Figure \ref{spatial disturbance figure}.  Figure \ref{spatial result} shows the optimal potential field obtained with the coarse problem and that obtained with the GS scheme at the first and tenth steps (initialized with the coarse field). Note that the coarse field error captures the global structure of the load field but misses the high frequencies, while the GS scheme corrects the high-frequency load imbalances.  In Figure \ref{fig:spatial} we again illustrate that the hierarchical scheme outperforms the decentralized GS scheme. 

\begin{figure}[!htb]
\begin{center}
\includegraphics[width=5in]{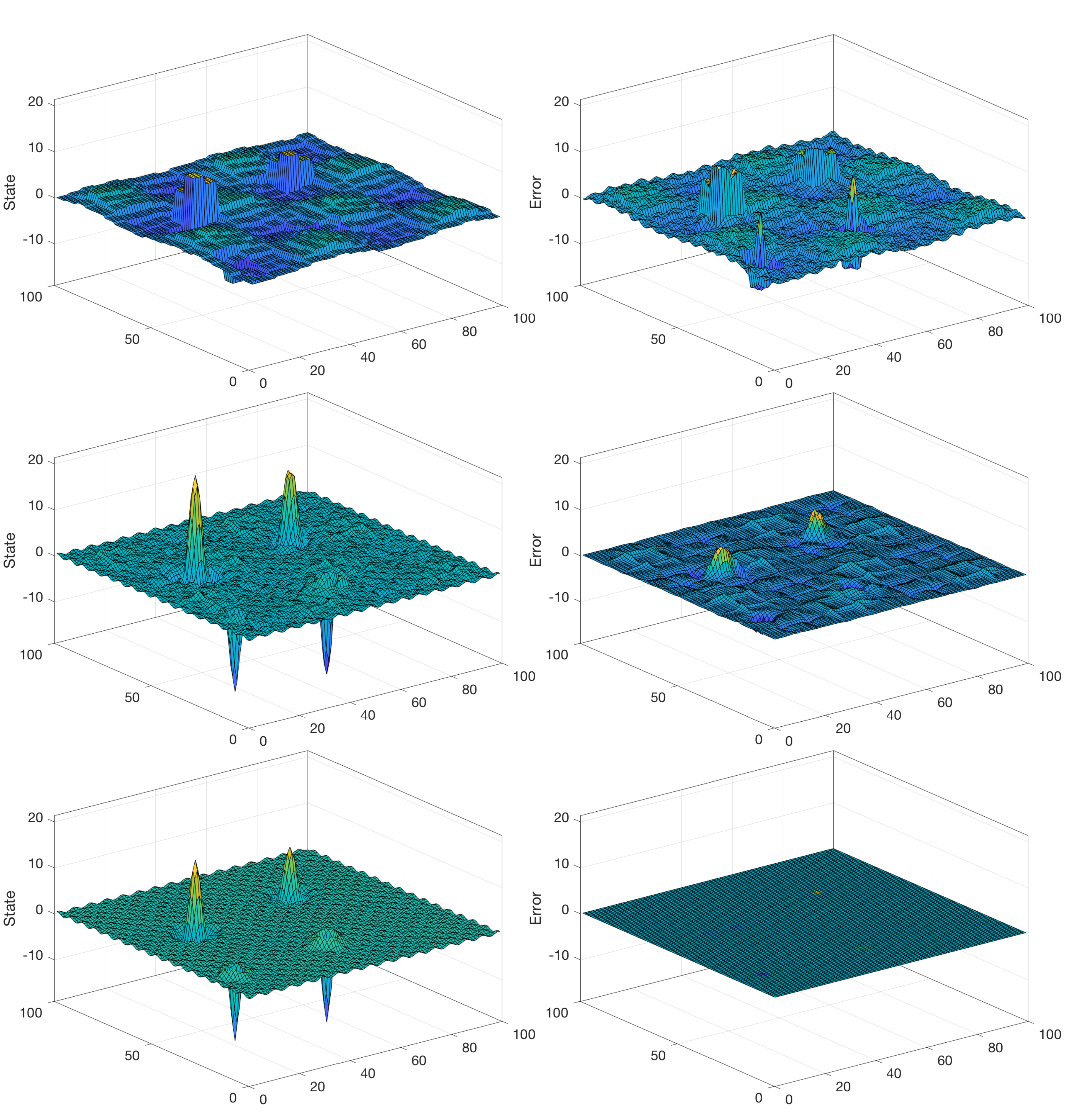}\caption{(Top) Potential field solution and error of coarse problem, (middle) solution and error of first GS update, and (bottom) solution and error of tenth GS update.}\label{spatial result}
\end{center}
\end{figure}

\begin{figure}[!htb]\begin{center}
\includegraphics[width=4.5in]{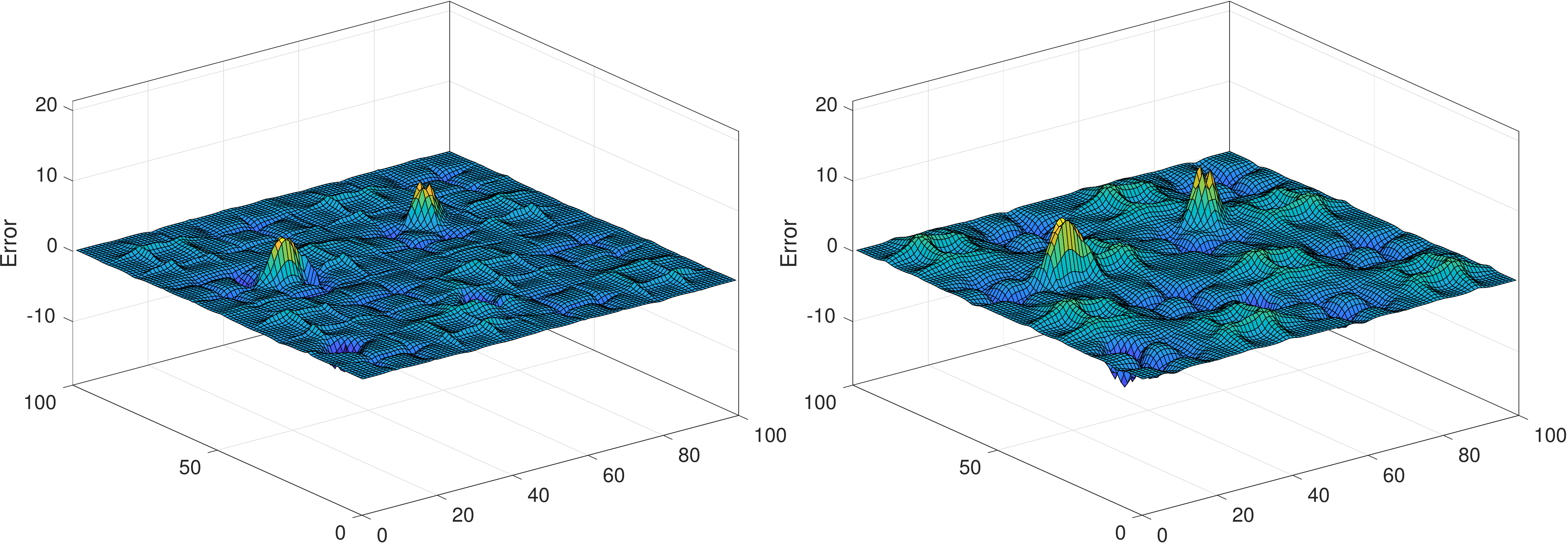}\caption{(Left) error for first GS update with coarsening and (right) error without coarsening.}\label{fig:spatial}
\end{center}
\end{figure}

\vspace{-0.2in}\subsection{Effect of Coarsening Strategy}

We now compare the efficiency of coarsening with different resolutions on the performance of GS. The results are given in Figure \ref{coarsening ordering figure} (top left and right) and reveal that using a higher resolution for the coarse problem does not necessarily result in better performance of the GS scheme. This is particularly evident in the temporal case while for the spatial case increasing the mesh resolution does help. We attribute this difference to the asymmetric nature of the temporal problem compared to the symmetric mesh of the spatial case.  For the temporal problem, we have found that the most effective coarsening strategy is to solve a sequence of coarse problems with increasing resolution. At each coarsening level, however, we only perform a single GS coordination step and the resulting coordination variables are used to initialize the GS coordination at the next level. The error evolution of this sequential coarsening scheme is shown in Figure \ref{coarsening ordering figure} (top left). We sequentially solved the coarse problems by using $M_c=1, M_c=2, M_c=4, M_c=5, M_c=10, M_c=20, M_c=25, M_c=50$ (this gives a total of eight GS steps). We note that, at the tenth GS step, the solution from this sequential coarsening scheme is about seventy times smaller than that obtained  with no coarsening. This can be attributed to the ability of the single step GS schemes to cover a wider range of frequencies. 
%

\vspace{-0.2in}\subsection{Effect of Coordination Order}
We solved the multi-scale temporal control problem \eqref{temporal GS problem} and the spatial control problem \eqref{spatial GS problem} with four different ordering methods for each problem. For temporal control problem, ordering method 1 was a lexicographic ordering, ordering method 2 was a reverse lexicographic ordering, ordering method 3 was a forward-backward ordering, and ordering method 4 was the red-black scheme. For the spatial control problem, ordering method 1 was a lexicographic ordering, ordering method 2 was a spiral-like ordering, ordering method 3 was the red-black ordering, and ordering method 4 was set by ordering the partitions based on the magnitude of the disturbance.  The results are presented in Figure \ref{coarsening ordering figure} (top left and right).  As can be seen, in temporal problem, the performance of reverse lexicographic ordering is significantly better than that achieved by other methods. This can be attributed to the asymmetry of the coupling topology. In particular, in the temporal problem, the primal variable information is propagated in forward direction while the dual information is propagated in reverse direction. It can be seen that dual information plays an important role in the convergence of the temporal problem.  In the spatial problem, the performance of the different orderings is virtually the same. The red-black ordering (which enables parallelism) achieves the same performance as the rest. We attribute this to the symmetry of the spatial domain. 

\begin{figure}[!htb]
\begin{center}
\includegraphics[width=4.5in]{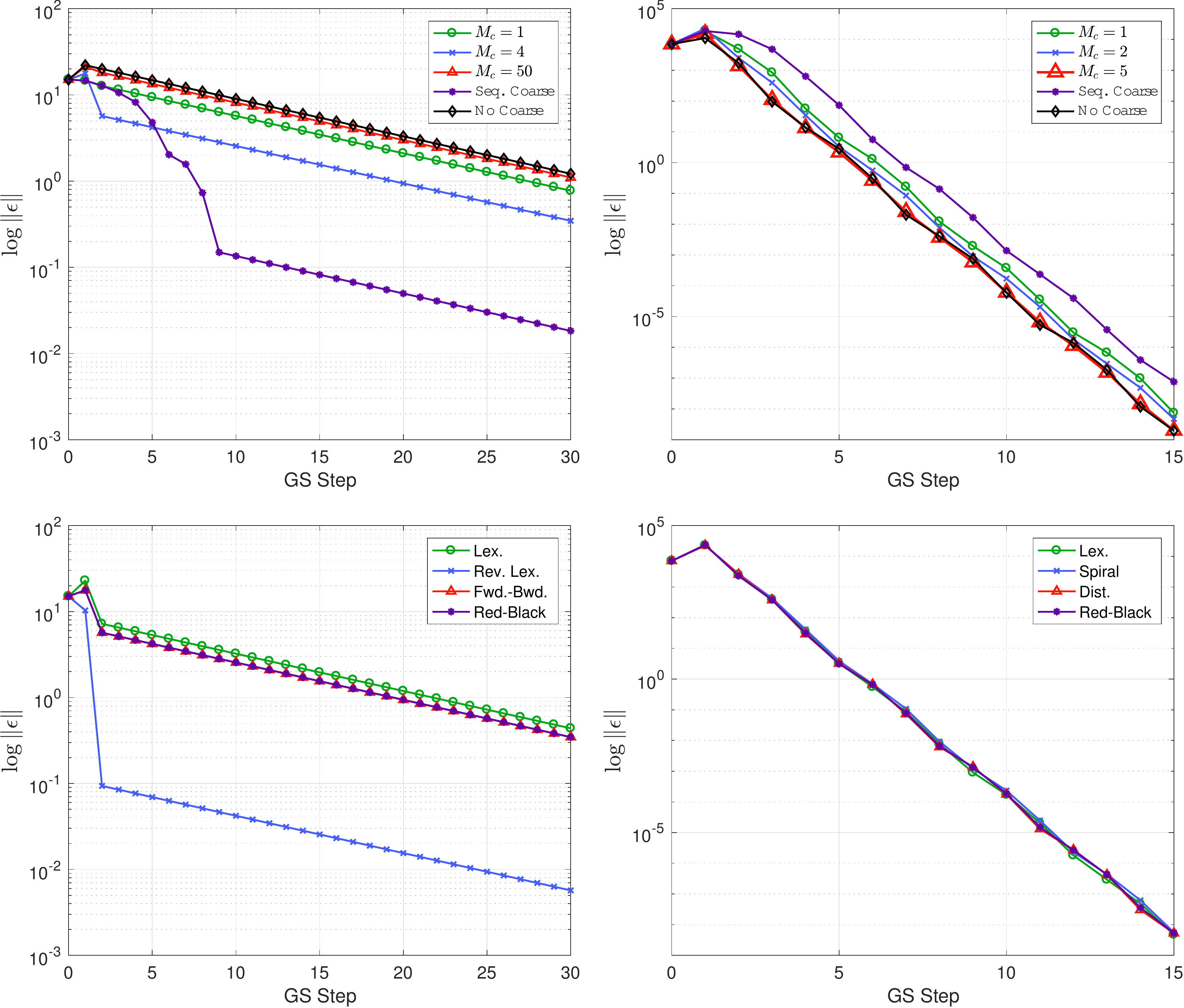}\caption{(Top-left) error of temporal control with different coarsening schemes, (top-right) error of spatial control with different coarsening schemes,  (bottom-left) error of temporal control with different ordering schemes, and (bottom-right) error of spatial control with different ordering schemes.}\label{coarsening ordering figure}
\end{center}
\end{figure}

\vspace{-0.2in}\section{Conclusions and Directions of Future Work}

We have presented basic elements of multi-grid computing schemes and illustrated how to use these to create hierarchical coordination architectures for complex systems. In particular, we discuss how Gauss-Seidel schemes can be seen as decentralized coordination schemes that handle high-frequency effects while coarse solution operators can be seen as low-resolution centralized coordination schemes that handle low-frequency effects. We believe that multi-grid provides a powerful framework to  systematically construct hierarchical coordination architectures but diverse challenges need to be addressed. In particular, it is important to understand convergence properties of GS schemes in more complex settings with nonlinear effects and inequality constraints. Moreover, it is necessary to develop effective coarsening (aggregation) schemes that can retain useful information while reducing complexity.  Moreover, it is desirable to combine hierarchical coordination schemes and existing control theory to analyze stability and robustness properties. 

\vspace{-0.2in}\section{Acknowledgements}

We acknowledge funding from the National Science Foundation under award NSF-EECS-1609183.  

\bibliography{vzavala}

\begin{thebibliography}{10}
\providecommand{\url}[1]{{#1}}
\providecommand{\urlprefix}{URL }
\expandafter\ifx\csname urlstyle\endcsname\relax
  \providecommand{\doi}[1]{DOI~\discretionary{}{}{}#1}\else
  \providecommand{\doi}{DOI~\discretionary{}{}{}\begingroup
  \urlstyle{rm}\Url}\fi

\bibitem{diehllift}
Albersmeyer, J., Diehl, M.: The lifted newton method and its application in
  optimization.
\newblock SIAM Journal on Optimization \textbf{20}(3), 1655--1684 (2010)

\bibitem{antoulas}
Antoulas, A.C., Sorensen, D.C., Gugercin, S.: A survey of model reduction
  methods for large-scale systems.
\newblock Contemporary mathematics \textbf{280}, 193--220 (2001)

\bibitem{arnold}
Arnold, M., Negenborn, R., Andersson, G., De~Schutter, B.: Distributed
  predictive control for energy hub coordination in coupled electricity and gas
  networks.
\newblock In: Intelligent Infrastructures, pp. 235--273. Springer (2010)

\bibitem{daoutidis}
Baldea, M., Daoutidis, P.: Control of integrated process networksÑa multi-time
  scale perspective.
\newblock Computers \& chemical engineering \textbf{31}(5), 426--444 (2007)

\bibitem{zavalaipopt}
Biegler, L.T., Zavala, V.M.: Large-scale nonlinear programming using ipopt: An
  integrating framework for enterprise-wide dynamic optimization.
\newblock Computers \& Chemical Engineering \textbf{33}(3), 575--582 (2009)

\bibitem{borzi2005multigrid}
Borz{\`\i}, A., Kunisch, K.: A multigrid scheme for elliptic constrained
  optimal control problems.
\newblock Computational Optimization and Applications \textbf{31}(3), 309--333
  (2005)

\bibitem{borzi}
Borz{\`\i}, A., Schulz, V.: Multigrid methods for pde optimization.
\newblock SIAM review \textbf{51}(2), 361--395 (2009)

\bibitem{boydadmm}
Boyd, S., Parikh, N., Chu, E., Peleato, B., Eckstein, J.: Distributed
  optimization and statistical learning via the alternating direction method of
  multipliers.
\newblock Foundations and Trends{\textregistered} in Machine Learning
  \textbf{3}(1), 1--122 (2011)

\bibitem{brandtalgebraic}
Brandt, A.: Algebraic multigrid theory: The symmetric case.
\newblock Applied mathematics and computation \textbf{19}(1), 23--56 (1986)

\bibitem{kroghdecentralized}
Camponogara, E., Jia, D., Krogh, B.H., Talukdar, S.: Distributed model
  predictive control.
\newblock Control Systems, IEEE \textbf{22}(1), 44--52 (2002)

\bibitem{chowpower}
Chow, J.H., Kokotovic, P.V.: Time scale modeling of sparse dynamic networks.
\newblock Automatic Control, IEEE Transactions on \textbf{30}(8), 714--722
  (1985)

\bibitem{diehlrealtime}
Diehl, M., Bock, H.G., Schl{\"o}der, J.P., Findeisen, R., Nagy, Z.,
  Allg{\"o}wer, F.: Real-time optimization and nonlinear model predictive
  control of processes governed by differential-algebraic equations.
\newblock Journal of Process Control \textbf{12}(4), 577--585 (2002)

\bibitem{farina2017hierarchical}
Farina, M., Zhang, X., Scattolini, R.: A hierarchical mpc scheme for
  interconnected systems.
\newblock arXiv preprint arXiv:1703.02739  (2017)

\bibitem{fisher}
Fisher, M.L.: The lagrangian relaxation method for solving integer programming
  problems.
\newblock Management science \textbf{50}(12\_supplement), 1861--1871 (2004)

\bibitem{deschutter}
Giselsson, P., Doan, M.D., Keviczky, T., De~Schutter, B., Rantzer, A.:
  Accelerated gradient methods and dual decomposition in distributed model
  predictive control.
\newblock Automatica \textbf{49}(3), 829--833 (2013)

\bibitem{luoadmm}
Hong, M., Luo, Z.Q.: On the linear convergence of the alternating direction
  method of multipliers.
\newblock arXiv preprint arXiv:1208.3922  (2012)

\bibitem{hierarbook}
Jamshidi, M., Tarokh, M., Shafai, B.: Computer-aided analysis and design of
  linear control systems.
\newblock Prentice-Hall, Inc. (1992)

\bibitem{illic}
Joo, J.Y., Ilic, M.D.: Multi-layered optimization of demand resources using
  lagrange dual decomposition.
\newblock Smart Grid, IEEE Transactions on \textbf{4}(4), 2081--2088 (2013)

\bibitem{kokotovicpower}
Kokotovic, P.: Subsystems, time scales and multimodeling.
\newblock Automatica \textbf{17}(6), 789--795 (1981)

\bibitem{kokotovic1982coherency}
Kokotovic, P., Avramovic, B., Chow, J., Winkelman, J.: Coherency based
  decomposition and aggregation.
\newblock Automatica \textbf{18}(1), 47--56 (1982)

\bibitem{frasch}
Kozma, A., Frasch, J.V., Diehl, M.: A distributed method for convex quadratic
  programming problems arising in optimal control of distributed systems.
\newblock In: Decision and Control (CDC), 2013 IEEE 52nd Annual Conference on,
  pp. 1526--1531. IEEE (2013)

\bibitem{hierarresid}
Lefort, A., Bourdais, R., Ansanay-Alex, G., Gu{\'e}guen, H.: Hierarchical
  control method applied to energy management of a residential house.
\newblock Energy and Buildings \textbf{64}, 53--61 (2013)

\bibitem{liuadmm}
Liu, C., Shahidehpour, M., Wang, J.: Application of augmented lagrangian
  relaxation to coordinated scheduling of interdependent hydrothermal power and
  natural gas systems.
\newblock Generation, Transmission \& Distribution, IET \textbf{4}(12),
  1314--1325 (2010)

\bibitem{deschutter2}
Negenborn, R.R., De~Schutter, B., Hellendoorn, J.: Efficient implementation of
  serial multi-agent model predictive control by parallelization.
\newblock In: Networking, Sensing and Control, 2007 IEEE International
  Conference on, pp. 175--180. IEEE (2007)

\bibitem{peponides}
Peponides, G.M., Kokotovic, P.V.: Weak connections, time scales, and
  aggregation of nonlinear systems.
\newblock Automatic Control, IEEE Transactions on \textbf{28}(6), 729--735
  (1983)

\bibitem{rawlingsunreachable}
Rawlings, J.B., Bonn{\'e}, D., J{\o}rgensen, J.B., Venkat, A.N., J{\o}rgensen,
  S.B.: Unreachable setpoints in model predictive control.
\newblock Automatic Control, IEEE Transactions on \textbf{53}(9), 2209--2215
  (2008)

\bibitem{rawlingsbook}
Rawlings, J.B., Mayne, D.: Model predictive control.
\newblock Noble Hill Publishing (2008)

\bibitem{scattoreview}
Scattolini, R.: Architectures for distributed and hierarchical model predictive
  control--a review.
\newblock Journal of Process Control \textbf{19}(5), 723--731 (2009)

\bibitem{scattohierar}
Scattolini, R., Colaneri, P.: Hierarchical model predictive control.
\newblock In: Decision and Control, 2007 46th IEEE Conference on, pp.
  4803--4808. IEEE (2007)

\bibitem{simon}
Simon, H.A., Ando, A.: Aggregation of variables in dynamic systems.
\newblock Econometrica: journal of the Econometric Society pp. 111--138 (1961)

\bibitem{rawlingscooperative}
Stewart, B.T., Venkat, A.N., Rawlings, J.B., Wright, S.J., Pannocchia, G.:
  Cooperative distributed model predictive control.
\newblock Systems \& Control Letters \textbf{59}(8), 460--469 (2010)

\bibitem{rawlingsjpc}
Stewart, B.T., Wright, S.J., Rawlings, J.B.: Cooperative distributed model
  predictive control for nonlinear systems.
\newblock Journal of Process Control \textbf{21}(5), 698--704 (2011)

\bibitem{mpcadmm}
Summers, T.H., Lygeros, J.: Distributed model predictive consensus via the
  alternating direction method of multipliers.
\newblock In: Communication, Control, and Computing (Allerton), 2012 50th
  Annual Allerton Conference on, pp. 79--84. IEEE (2012)

\bibitem{zavalamultigrid}
Zavala, V.M.: New architectures for hierarchical predictive control.
\newblock IFAC-PapersOnLine \textbf{49}(7), 43--48 (2016)

\bibitem{zavalage}
Zavala, V.M., Anitescu, M.: Real-time nonlinear optimization as a generalized
  equation.
\newblock SIAM Journal on Control and Optimization \textbf{48}(8), 5444--5467
  (2010)

\bibitem{zavala2010real}
Zavala, V.M., Anitescu, M.: Real-time nonlinear optimization as a generalized
  equation.
\newblock SIAM Journal on Control and Optimization \textbf{48}(8), 5444--5467
  (2010)

\bibitem{zavalaasnmpc}
Zavala, V.M., Biegler, L.T.: The advanced-step nmpc controller: Optimality,
  stability and robustness.
\newblock Automatica \textbf{45}(1), 86--93 (2009)

\bibitem{hug}
Zhu, D., Yang, R., Hug-Glanzmann, G.: Managing microgrids with intermittent
  resources: A two-layer multi-step optimal control approach.
\newblock In: North American Power Symposium (NAPS), 2010, pp. 1--8. IEEE
  (2010)

\end{thebibliography}

\clearpage

\end{document}